\documentclass[a4paper,UKenglish,cleveref, autoref, thm-restate]{lipics-v2021}
\usepackage{graphicx}
\usepackage{mathtools}
\usepackage{leftindex}

\allowdisplaybreaks
\nolinenumbers

\newcommand{\Z}{\mathbb{Z}}
\newcommand{\N}{\mathbb{N}}

\newcommand{\R}{\mathbb{R}}

\newcommand{\K}{\mathbb{K}}
\newcommand{\F}{\mathbb{F}}

\newcommand{\SL}{\mathsf{SL}}
\newcommand{\GL}{\mathsf{GL}}

\newcommand{\diag}{\operatorname{diag}}
\newcommand{\Ann}{\operatorname{Ann}}

\newcommand{\mC}{\mathcal{C}}

\newcommand{\mY}{\mathcal{Y}}
\newcommand{\mF}{\mathcal{F}}

\newcommand{\tY}{\widetilde{\mathcal{Y}}}

\newcommand{\mA}{\mathcal{A}}

\newcommand{\mZ}{\mathcal{Z}}

\newcommand{\ba}{\boldsymbol{a}}
\newcommand{\bb}{\boldsymbol{b}}

\newcommand{\bh}{\boldsymbol{h}}
\newcommand{\bz}{\boldsymbol{z}}

\newcommand{\bzer}{\boldsymbol{0}}

\newcommand{\gen}[1]{\langle {#1} \rangle}

\newcommand{\oX}{\mkern 1.5mu\overline{\mkern-1.5mu X \mkern-1.5mu}\mkern 1.5mu}

\newcommand{\ox}{\overline{x}}
\newcommand{\oA}{\overline{A}}
\newcommand{\frp}{\mathfrak{p}}

\newtheorem{thrm}{Theorem}[section]
\newtheorem{lem}[thrm]{Lemma}
\newtheorem{prop}[thrm]{Proposition}
\newtheorem{cor}[thrm]{Corollary}
\theoremstyle{definition}
\newtheorem{defn}[thrm]{Definition}
\theoremstyle{definition}
\newtheorem{exmpl}[thrm]{Example}
\theoremstyle{definition}
\newtheorem{rmk}[thrm]{Remark}

\title{Submonoid Membership in n-dimensional lamplighter groups and S-unit equations}
\author{Ruiwen Dong}{Department of Mathematics, Saarland University, Germany \and Magdalen College, University of Oxford, United Kingdom}{ruiwen.dong@magd.ox.ac.uk}{https://orcid.org/0009-0007-4349-082X}{}

\authorrunning{ }

\Copyright{ } 



\category{} 

\relatedversion{} 

\supplement{}

\funding{This work was funded by the European Union (ERC, AdG grant 101097307). Views and opinions expressed are however those of the author(s) only and do not necessarily reflect those of the European Union or the European Research Council. Neither the European Union nor the granting authority can be held responsible for them.}

\acknowledgements{The author would like to thank Doron Shafrir for pointing out the open problem on Submonoid Membership in finite extension of groups, as well as for discussion of his work~\cite{shafrir2024decidability}. The author would like to thank the anonymous reviewers for their helpful comments.}

\begin{document}

\maketitle
\begin{abstract}
We show that Submonoid Membership is decidable in $n$-dimensional lamplighter groups $(\mathbb{Z}/p\mathbb{Z}) \wr \mathbb{Z}^n$ for any prime $p$ and integer $n$. 
More generally, we show decidability of Submonoid Membership in semidirect products of the form $\mathcal{Y} \rtimes \mathbb{Z}^n$, where $\mathcal{Y}$ is any finitely presented module over the Laurent polynomial ring $\mathbb{F}_p[X_1^{\pm}, \ldots, X_n^{\pm}]$.
Combined with a result of Shafrir (2024), this gives the first example of a group $G$ and a finite index subgroup $\widetilde{G} \leq G$, such that Submonoid Membership is decidable in $\widetilde{G}$ but undecidable in $G$.

To obtain our decidability result, we reduce Submonoid Membership in $\mathcal{Y} \rtimes \mathbb{Z}^n$ to solving S-unit equations over $\mathbb{F}_p[X_1^{\pm}, \ldots, X_n^{\pm}]$-modules. We show that the solution set of such equations is effectively $p$-automatic, extending a result of Adamczewski and Bell (2012).
As an intermediate result, we also obtain that the solution set of the Knapsack Problem in $\mathcal{Y} \rtimes \mathbb{Z}^n$ is effectively $p$-automatic.
\end{abstract}

\section{Introduction}
\paragraph*{Algorithmic problems in infinite groups}
Algorithmic problems concerning groups are a classical topic in algebra and theoretical computer science.
Dating back to the work of Max Dehn from the 1910s, this area has traditionally served as a bridge between geometry, algebra and logic, and has now found numerous applications in automata theory, program analysis and complexity theory~\cite{beals1993vegas, blondel2005decidable, derksen2005quantum, hrushovski2018polynomial}.
One of the central decision problems on groups is the \emph{Submonoid Membership} problem: given a finite number of elements $g_1, g_2, \ldots, g_{k}$ in a group $G$ and an element $g \in G$, does $g$ belong to the submonoid generated by $g_1, g_2, \ldots, g_{k}$?
In the seminal work of Markov from the 1950s~\cite{markov1951certain}, it was shown that Submonoid Membership is undecidable for a matrix group of dimension six.
The undecidability result was later strengthened to the group $\SL(4, \Z)$ of $4 \times 4$ integer matrices of determinant one~\cite{mikhailova1958occurrence}.
Subsequently, a number of positive decidability results were obtained, such as for commutative matrix groups~\cite{babai1996multiplicative}, low-dimensional matrix groups~\cite{DBLP:conf/icalp/ColcombetOS019, potapov2017decidability}, and certain graph groups~\cite{lohrey2008submonoid}.
See~\cite{dong2023recent, lohrey2024membership} for recent surveys.

A generalization of Submonoid Membership is the \emph{Rational Subset Membership} problem: given a rational subset $S$ of a group $G$ and an element $g \in G$, does $g$ belong to $S$?
Here, a subset $S$ of a group $G$ is called a \emph{rational subset} if there is a finite alphabet $\Sigma$, a monoid homomorphism $\varphi \colon \Sigma^* \rightarrow G$, and a regular language $L \subseteq \Sigma^*$, such that $\varphi(L) = S$.
In particular, if $L = \Sigma^*$, then we recover the definition of Submonoid Membership.
In fact, many decidability results for Submonoid Membership are also proven for the more general Rational Subset Membership, such as in free groups~\cite{benois1969parties}, abelian groups~\cite{grunschlag1999algorithms}, and certain wreath products~\cite{lohrey2015rational}.
Nevertheless, Bodart~\cite{bodart2024membership} and Shafrir\footnote{An example appeared in the Bachelor's thesis of Potthast~\cite{potthast2020submonoid}, based on an unpublished draft of Shafrir.} recently provided examples of groups with undecidable Rational Subset Membership but decidable Submonoid Membership.

A longstanding open problem in computational group theory is whether decidability of Submonoid Membership is preserved under finite extension of groups.
Grunschlag~\cite{grunschlag1999algorithms} showed that for any group $G$ and a finite index subgroup $\widetilde{G} \leq G$, \emph{Rational Subset Membership} is decidable in $\widetilde{G}$ if and only if it is decidable in $G$.
The same statement is also true for the \emph{Subgroup Membership} problem~\cite{grunschlag1999algorithms} (given $g_1, g_2, \ldots, g_{k}, g$, does $g$ belong to the subgroup generated by $g_1, g_2, \ldots, g_{k}$?).
Whether the same holds for \emph{Submonoid Membership} remained open.
One of the goals of this paper is to construct a counterexample to this open problem: we construct a group $G$ and a finite index subgroup $\widetilde{G} \leq G$, such that Submonoid Membership is decidable in $\widetilde{G}$ but undecidable in $G$.

\paragraph*{n-dimensional lamplighter groups}

In this paper, we study Submonoid Membership in \emph{$n$-dimensional lamplighter groups} as well as their generalizations.
The lamplighter groups are well-studied objects in the context of geometric group theory~\cite{grigorchuk2001lamplighter, silva2005class}.
Given an integer $p \geq 2$ and $n \in \N$, the $n$-dimensional lamplighter group, denoted $(\Z/p\Z) \wr \Z^n$, can be most easily defined as the (multiplicative) $2 \times 2$ upper-triangular matrix group
\begin{equation}\label{eq:deflamp}
\left\{
\begin{pmatrix}
    X_1^{a_1} X_2^{a_2} \cdots X_n^{a_n} & f \\
    0 & 1\\
\end{pmatrix}
\;\middle|\;
a_1, a_2, \ldots, a_n \in \Z, \; f \in (\Z/p\Z)[X_1, X_1^{-1}, \ldots, X_n, X_n^{-1}]
\right\},
\end{equation}
where $(\Z/p\Z)[X_1, X_1^{-1}, \ldots, X_n, X_n^{-1}]$ is the Laurent polynomial ring over $n$ variables, with coefficients in the finite ring $\Z/p\Z = \{0, 1, \ldots, p-1\}$.
When $p$ is prime, $\Z/p\Z$ is actually the finite field $\F_p$.
Intuitively, an element in $(\Z/p\Z) \wr \Z^n$ can be understood as the following configuration: at each position in the lattice $\Z^n$ there is a \emph{lamp}, each lamp has a \emph{state} among $\Z/p\Z = \{0, 1, \ldots, p-1\}$, and there is a \emph{pointer} at some position in $\Z^n$.
In particular, for the element 
$
\begin{pmatrix}
    X_1^{a_1} \cdots X_n^{a_n} & f \\
    0 & 1\\
\end{pmatrix}
$
, the pointer is at position $(a_1, \ldots, a_n) \in \Z^n$, the lamp at position $(z_1, \ldots, z_n) \in \Z^n$ has state $s$ if the monomial $X_1^{z_1} \cdots X_n^{z_n}$ appearing in $f$ has coefficient $s$. In particular, all but finitely many lamps have state 0.
Multiplication in the group $(\Z/p\Z) \wr \Z^n$, like matrix multiplication, corresponds to first moving the pointer by a vector $(a_1, \ldots, a_n)$, then additively changing the states of the lamps around the new location.

Owing to this geometric interpretation, the study of lamplighter groups has close connections to automata theory, tiling problems, random walks, and graph theory~\cite{barbieri2016group, bartholdi2022simulations, kambites2006spectra, lohrey2011tilings, lyons1996random}.
Lamplighter groups are also important examples of \emph{metabelian groups}, which are some of the most tractable generalization of abelian groups.
Recall that a group $G$ is called \emph{metabelian} if it admits a normal subgroup $A$ such that both $A$ and $G/A$ are abelian.
An important amount of research in computational group theory concentrates on metabelian groups, since decision problems in abelian groups are already well-understood.
For example, Cadilhac, Chistikov and Zetzsche~\cite{DBLP:conf/icalp/CadilhacCZ20} showed decidability of Rational Subset Membership in the \emph{Baumslag-Solitar groups} $\mathsf{BS}(1, p)$; Lohrey, Steinberg and Zetzsche~\cite{lohrey2011tilings} showed decidability of Rational Subset Membership in $1$-dimensional lamplighter groups.
On the other hand, the $2$-dimensional lamplighter groups occupy an interesting position as they have undecidable Rational Subset Membership~\cite{lohrey2011tilings} but decidable Submonoid Membership~\cite{potthast2020submonoid}.
The decision algorithm in~\cite{potthast2020submonoid} cannot be generalized to dimensions larger than 2 as it relies on a reduction to Rational Subset Membership in dimension 1.
It is thus left as an open problem whether Submonoid Membership is decidable in lamplighter groups of dimension $n \geq 3$.

\paragraph*{Contributions of this paper}

In this paper, we show that Submonoid Membership is decidable in lamplighter groups $(\Z/p\Z) \wr \Z^n$ of every dimension $n \in \N$ and \emph{prime} $p$.
In fact, we prove a much more general result (Theorem~\ref{thm:main}): Submonoid Membership is decidable for semidirect products $\mY \rtimes \Z^n$, where $\mY$ is a finitely presented module over the Laurent polynomial ring $\F_p[X_1^{\pm}, \ldots, X_n^{\pm}]$.

An important consequence of the more general Theorem~\ref{thm:main} is the resolution of a longstanding open problem (Corollary~\ref{cor:finext}).
Combined with a recent result of Shafrir~\cite{shafrir2024decidability}, we show that there exist a group $G$ and a finite index subgroup $\widetilde{G} \leq G$, such that Submonoid Membership decidable in $\widetilde{G}$ but undecidable in $G$.

Our main strategy to decide Submonoid Membership in $\mathcal{Y} \rtimes \Z^n$ is to reduce it to solving \emph{S-unit equations} over modules: these are equations of the form $x_1 c_1 + x_2 c_2 + \cdots + x_m c_m = c_0$, where the variables $x_i$ take value in a multiplicative subgroup of a commutative ring.
We show that the solutions of such S-unit equations over an $\F_p[X_1^{\pm}, \ldots, X_n^{\pm}]$-module forms an effectively $p$-automatic set (Theorem~\ref{thm:Sunit}), generalizing a result of Adamczewski and Bell~\cite{adamczewski2012vanishing} (see also the independent work of Derksen and Masser~\cite{derksen2012linear} for a similar result).
As an intermediate result, we also obtain effective $p$-automaticity for the solution set of the \emph{Knapsack Problem} in $\mathcal{Y} \rtimes \Z^n$: given $g_1, g_2, \ldots, g_k, g \in \mathcal{Y} \rtimes \Z^n$, find $(z_1, \ldots, z_k) \in \Z^k$ such that $g_1^{z_1} g_2^{z_2} \cdots g_k^{z_k} = g$.
We give an example (Example~\ref{expl:KPtoSunit}) where the solution set is $p$-automatic but not semilinear.
This shows that the situation in semidirect products is much more complicated than in the wreath product of two abelian groups, as seen in the work of Ganardi, K\"{o}nig, Lohrey, and Zetzsche~\cite{DBLP:conf/stacs/GanardiKLZ18}.

Note that our decidability results heavily rely on working over the finite ring $\Z/p\Z$ instead of the infinite ring $\Z$.
Indeed, Submonoid Membership is known to be undecidable~\cite{lohrey2015rational} in $\Z \wr \Z^n$ for $n \geq 1$, and solving S-unit equations over $\Z[X_1^{\pm}]$-modules is also known to be undecidable~\cite{dong2024linear}.
Nevertheless, it is not clear whether the decidability for $(\Z/p\Z) \wr \Z^n$ can be generalized to non-prime numbers $p$, for example when $p = p_1 p_2$ for two primes $p_1 \neq p_2$. This is closely related to deciding emptiness of the intersection of a $p_1$-automatic set and a $p_2$-automatic set, which to the best of our knowledge remains an open problem (however, see~\cite{Hieronymi2022, karimov2024decidability} for recent progress).
Progress towards a solution for non-prime $p$ will also be the first step towards a classification of metabelian groups with decidable Submonoid Membership, namely by tackling those with torsion commutators.

\section{Preliminaries}

\paragraph*{Laurent polynomial ring and modules}

Let $p$ be a prime number, denote by $\F_p$ the finite field of size $p$.
Denote by $\F_p[X_1^{\pm}, \ldots, X_n^{\pm}]$ the Laurent polynomial ring over $\F_p$ with $n$ variables: this is the set of polynomials over the variables $X_1, X_1^{-1}, \ldots, X_n, X_n^{-1}$, with coefficients in $\F_p$, such that $X_i X_i^{-1} = 1$ for all $i$.
When $n$ is fixed, we denote by $\oX$ the tuple of variables $(X_1, \ldots, X_n)$, and $\F_p[X_1^{\pm}, \ldots, X_n^{\pm}]$ is written in short as $\F_p[\oX^{\pm}]$.

For a vector $\ba = (a_1, \ldots, a_n) \in \Z^n$, denote by $\oX^{\ba}$ the monomial $X_1^{a_1} X_2^{a_2} \cdots X_n^{a_n}$.
When $\Lambda$ is a subgroup of $\Z^n$, we denote by $\F_p[\oX^{\Lambda}] \subseteq \F_p[\oX^{\pm}]$ the subring of polynomials of the form $\sum_{\ba \in \Lambda} c_{\ba} \oX^{\ba}$.

Since we work in polynomial rings over a finite field $\F_p$, it is worth pointing out a useful equation, folklorishly dubbed ``the freshman's dream''. For any $f \in \F_p[\oX^{\pm}]$, we have $(f + 1)^p = f^p + 1$, and more generally $(f + 1)^{p^k} = f^{p^k} + 1$ for all $k \in \N$.

When $R$ is a commutative ring (such as $\F_p[\oX^{\pm}]$), an \emph{$R$-module} is defined as an abelian group $(M, +)$ along with an operation $\cdot \;\colon R \times M \rightarrow M$ satisfying $f \cdot (m+m') = f \cdot m + f \cdot m'$, $(f + g) \cdot m = f \cdot m + g \cdot m$, $fg \cdot m = f \cdot (g \cdot m)$ and $1 \cdot m = m$.
For example, for any $d \in \N$, $\F_p[\oX^{\pm}]^d$ is an $\F_p[\oX^{\pm}]$-module by $f \cdot (h_1, \ldots, h_d) = (fh_1, \ldots, fh_d)$.
We often use a bold symbol $\bh$ to denote a vector $(h_1, \ldots, h_d) \in \F_p[\oX^{\pm}]^d$.

Given elements $\bh_1, \ldots, \bh_k$ in an $R$-module $M$, we say that they \emph{generate} the $R$-module
\[ \sum_{i=1}^k R \cdot \bh_i \coloneqq \left\{\sum_{i=1}^k r_i \cdot \bh_i \;\middle|\; r_1, \ldots, r_k \in R \right\} \subseteq M. \]
Given submodules $M, M'$ of $\F_p[\oX^{\pm}]^d$ such that $M \supseteq M'$, we define the quotient $M/M' \coloneqq \{\overline{m} \mid m \in M\}$ where $\overline{m}_1 = \overline{m}_2$ if and only if $m_1 - m_2 \in M'$.
This quotient is also an $\F_p[\oX^{\pm}]$-module.
We say that an $\F_p[\oX^{\pm}]$-module $\mY$ is \emph{finitely presented} if it can be written as a quotient $\F_p[\oX^{\pm}]^d/M$ for a finitely generated submodule $M$ of $\F_p[\oX^{\pm}]^d$ for some $d \in \N$.
We call a \emph{finite presentation} of $\mY$ the generators of such $M$.
Every finitely generated $\F_p[\oX^{\pm}]$-module admits a finite presentation.

Given a finitely presented $\F_p[X_1^{\pm}, \ldots, X_n^{\pm}]$-module $\mY$, we can define the following group using \emph{semidirect product}:
\begin{equation}\label{eq:defsemi}
\mY \rtimes \Z^n \coloneqq \{(y, \ba) \mid y \in \mY, \ba \in \Z^n\};
\end{equation}
multiplication and inversion in this group are defined by
\begin{equation}\label{eq:defsemi2}
(y, \ba) \cdot (y', \ba') = \big(y + \oX^{\ba} \cdot y', \ba + \ba'\big), \quad (y, \ba)^{-1} = \big(\! - \oX^{-\ba} \cdot y, -\ba\big).
\end{equation}
The neutral element of $\mY \rtimes \Z^n$ is $(0, \bzer)$, where $\bzer$ denotes the zero vector in $\Z^n$.
Intuitively, the element $(y, \ba)$ can be seen as a $2 \times 2$ matrix
$
\begin{pmatrix}
\oX^{\ba} & y \\
0 & 1 \\
\end{pmatrix}
$, where group multiplication is represented by matrix multiplication.
We have the formula $(f, \ba)(y, \bzer)(f, \ba)^{-1} = (\oX^{\ba} \cdot y, \bzer)$.
Note that $\mY \rtimes \Z^n$ naturally contains the subgroup $\mY \cong \{(y, \bzer) \mid y \in \mY\}$, and there is a projection map $\pi \colon \mY \rtimes \Z^n \rightarrow \Z^n, \, (y, \ba) \mapsto \ba$, such that $\ker(\pi) = \mY$.

In particular, if we take $\mY = \F_p[\oX^{\pm}]$ considered as a $\F_p[\oX^{\pm}]$-module, then we recover the definition~\eqref{eq:deflamp} of the $n$-dimensional lamplighter group $(\Z/p\Z) \wr \Z^n$.

\paragraph*{Effective computation in modules}
In order to effectively compute in groups of the form $\mY \rtimes \Z^n$, we need to be able to effectively compute in the $\F_p[\oX^{\pm}]$-module $\mY$.
The following are some classic problems with effective algorithms that we will make use of.
We state the following result over $\F_p[\oX^{\pm}]$, although they hold over polynomial rings over any field.

\begin{lem}[{\cite[Lemma~2.1, 2.2]{baumslag1981computable}}]\label{lem:classicdec}
    Let $\mY$ be an $\F_p[\oX^{\pm}]$-module with a given finite presentation.
    The following problems are effectively solvable:
    \begin{enumerate}[(i)]
        \item \textit{(Submodule Membership)} Given elements $y_1, \ldots, y_k, y \in \mY$, decide whether $y$ is in the submodule generated by $y_1, \ldots, y_k$.
        \item \textit{(Submodule Presentation)} Given elements $y_1, \ldots, y_k \in \mY$, compute a finite presentation for the submodule generated by $y_1, \ldots, y_k$.
        \item \textit{(Computing intersection and sum)} Given the generators of a submodule $A \subseteq \mY$ and the generators of a submodule $B \subseteq \mY$, compute the generators for the submodule $A \cap B$ and for the submodule $A + B = \{a + b \mid a \in A, b \in B\}$.
        \item \textit{(Computing quotient)} Given the generators of a submodule $\mY' \subseteq \mY$, compute a finite presentation of the quotient $\mY/\mY'$.
    \end{enumerate}
\end{lem}

Let $\Lambda$ be a subgroup of $\Z^n$, then any $\F_p[\oX^{\pm}]$-module can be considered as a (possibly infinitely generated) $\F_p[\oX^{\Lambda}]$-module. For example, the $\F_p[X_1^{\pm}, X_2^{\pm}]$-module $\F_p[X_1^{\pm}, X_2^{\pm}]$ can be considered as an $\F_p[X_1^{\pm}]$-module, but it is not finitely generated (it is generated by the infinite set $\{\ldots, X_2^{-2}, X_2^{-1}, 1, X_2, X_2^2, \ldots\}$).
However, finitely generated $\F_p[\oX^{\Lambda}]$-submodules of an $\F_p[\oX^{\pm}]$-module are still effectively computable:

\begin{lem}[{\cite[Theorem~2.14]{baumslag1981computable} or \cite[Theorem~2.6]{baumslag1994algorithmic}}]\label{lem:sgmod}
    Let $\Lambda$ be a subgroup of $\Z^n$, and $\mY$ be a finitely presented $\F_p[\oX^{\pm}]$-module.
    Given a finite subset $S$ of $\mY$, one can effectively compute the finite presentation of the $\F_p[\oX^{\Lambda}]$-module $\sum_{s \in S} \F_p[\oX^{\Lambda}] \cdot s$.
\end{lem}

\paragraph*{p-automatic sets}
We recall the standard notion of \emph{$p$-automatic subsets} of $\Z$, and more generally, $p$-automatic subsets of $\Z^d$.

Let $\Sigma$ be a finite alphabet. An automaton over $\Sigma$ is a tuple $\mA = (Q, \Sigma, \delta, q_I, F)$,
where $Q$ is a finite set of states, $\delta \colon Q \times \Sigma \rightarrow Q$ is the transition function, $q_I$ is the initial state, and $F \subseteq Q$ is the set of final states.
For a state $q$ in $Q$ and for a finite word $w = w_1 w_2 \cdots w_n$ over the alphabet $\Sigma$, we define $\delta(q, w)$ recursively by $\delta(q, w) = \delta(\delta(q, w_1 w_2 \cdots w_{n-1}), w_n)$. 
The word $w$ is \emph{accepted} by $\mA$ if $\delta(q_I, w) \in F$.

Let $p \geq 2$ be an integer, define the alphabet $\Sigma_p \coloneqq \{+, -, 0, 1, 2, \ldots, p-1\}$.
Let $n \geq 0$ be a non-negative integer and let $w_r w_{r-1} \cdots w_1 w_0 \in (\Sigma_p)^{r+1}$ be the base-$p$ expansion of $n$. 
That is, $n = \sum_{i = 0}^r w_i p^i$ with the smallest possible $r$. 
We denote by $w(n)$ the word $+ w_0 w_1 \cdots w_r$.
For example, if $p = 2$, then $w(4) = +001$.
Similarly, for a negative integer $-n$, let $w(-n)$ denote the word $- w_0 w_1 \cdots w_r$. 
A subset $S$ of $\Z$ is called \emph{$p$-automatic} if there is an automaton over $\Sigma_p$ that accepts exactly the set $w(S) \coloneqq \{w(s) \mid s \in S\}$.
For example, the set $S = \{2^k \mid k \in \N\}$ is 2-automatic, because $w(S) = \{+1, +01, +001, +0001, \cdots \}$ is accepted by an automaton over $\Sigma_2 = \{+, -, 0, 1\}$.

Let $d \in \N$. The definition of $p$-automatic subsets of $\Z$ can be naturally generalized to $p$-automatic subsets of $\Z^d$.
For $(z_1, \ldots, z_d) \in \Z^d$, consider the tuple of words $(w(z_1), \ldots, w(z_d))$ over the alphabet $\Sigma_p$.
We naturally identify a $d$-tuple of words $(w(z_1), \ldots, w(z_d))$ over the alphabet $\Sigma_p$ as a single word $w(z_1, \ldots, z_d)$ over the alphabet $\Sigma_p^d$: if the words $w(z_1), \ldots, w(z_d)$ have different lengths, we add a minimal number of zeros at the end so they all have equal lengths.
Naturally, a subset $S$ of $\Z^d$ is called \emph{$p$-automatic} if there is an automaton over $\Sigma_p^d$ that accepts exactly the set $w(S) \coloneqq \{w(s) \mid s \in S\}$.
For example, the set $S = \{(a, 2a) \mid a < -1\}$ $\subset \Z^2$ is 2-automatic, because 
\[
w(S) = \{(-, -)(a_1, 0)(a_2, a_1)(a_3, a_2) \cdots (a_k, a_{k-1}) (1, a_k) (0, 1) \mid k \in \N, a_1, \ldots, a_k \in \{0, 1\}\}
\]
is accepted by an automaton over $\Sigma_2^2$.
We say a subset $S \subseteq \Z^d$ is \emph{effectively} $p$-automatic if its accepting automaton is explicitly given.
$p$-automatic sets enjoy various closure properties.
In this paper, we will use the following:

\begin{lem}[{\cite{wolper2000construction}}]\label{lem:pauto}
    Let $p \geq 2$ be an integer.
    \begin{enumerate}[(1)]
        \item If $S$ and $T$ are effectively $p$-automatic subsets of $\Z^d$, then $S \cap T$ and $S \cup T$ are also effectively $p$-automatic.
        \item If $S \subseteq \Z^d$ is effectively $p$-automatic, and $\varphi \colon \Z^d \rightarrow \Z^n$ is a linear transformation, then $\varphi(S) \subseteq \Z^n$ is also effectively $p$-automatic.
        \item If $S \subseteq \Z^d$ is effectively $p$-automatic, and $\phi \colon \Z^N \rightarrow \Z^d$ is a linear transformation, then $\phi^{-1}(S) \coloneqq \{v \in \Z^N \mid \phi(v) \in S\}$ is also effectively $p$-automatic.
        \item It is decidable whether a given effectively $p$-automatic set $S$ is empty.
    \end{enumerate}
\end{lem}

\section{Main results and consequences}\label{sec:mainres}

The main result of this paper is the following.

\begin{restatable}{thrm}{thmmain}\label{thm:main}
    Let $p$ be a prime number, $\oX = (X_1, \ldots, X_n)$ be a tuple of variables and $\mY$ be a finitely presented $\F_p[\oX^{\pm}]$-module.
    Then, Submonoid Membership is decidable in the group $\mY \rtimes \Z^n$.
\end{restatable}

In particular, taking $\mY = \F_p[\oX^{\pm}]$, we immediately obtain decidability of Submonoid Membership in the $n$-dimensional lamplighter groups $(\Z/p\Z) \wr \Z^n$ for prime $p$. 

We say a group is \emph{virtually abelian} if it has a finite index subgroup that is abelian.
A recent result of Shafrir~\cite[Theorem~8.1]{shafrir2024decidability} shows that membership problem in a \emph{fixed} rational subset $S$ of a group $L$ can be reduced to the membership problem in a fixed submonoid $M$ of the group $L \times H$, where $H$ is an explicitly constructed finitely generated virtually abelian group.
Combined with Theorem~\ref{thm:main}, this allows us to construct an example where decidability of Submonoid Membership is not preserved under finite extension of groups:
\begin{cor}\label{cor:finext}
    There exist a group $G$ and a finite index subgroup $\widetilde{G} \leq G$, such that Submonoid Membership is decidable in $\widetilde{G}$ but undecidable in $G$.
\end{cor}
\begin{proof}
    By~\cite[Theorem~10]{lohrey2011tilings}, there is a rational subset $S$ of the two-dimensional lamplighter group $L \coloneqq (\Z/2\Z) \wr \Z^2$, whose membership problem (given $g \in L$, whether $g \in S$?) is undecidable.
    By Shafrir's result~\cite[Theorem~8.1]{shafrir2024decidability}, the membership problem in $S$ reduces to the membership problem in a finitely generated submonoid of the group $G \coloneqq L \times H$, where $H$ is finitely generated virtually abelian.
    Therefore, Submonoid Membership is undecidable in $G = L \times H$. 
    
    Since $H$ is finitely generated virtually abelian, it has a finite index subgroup $\Z^d \leq H$ for some $d$.
    Let $\widetilde{G} \coloneqq L \times \Z^d$, so the index $[G \colon \widetilde{G}] = [H \colon \Z^d]$ is finite.
    We now show that Submonoid Membership is decidable in $\widetilde{G}$.
    
    Note that $\widetilde{G} = L \times \Z^d = \left((\Z/2\Z) \wr \Z^2 \right) \times \Z^d$ can be written as a semidirect product $\mY \rtimes \Z^{d+2}$, where $\mY \coloneqq \F_2[X_1^{\pm}, X_2^{\pm}]$ is considered as an $\F_2[X_1^{\pm}, X_2^{\pm}, X_3^{\pm} \ldots, X_{d+2}^{\pm}]$-module where the elements $X_3, \ldots, X_{d+2}$ act trivially. That is, $\mY = \F_2[X_1^{\pm}, \ldots, X_{d+2}^{\pm}]/\sum_{i=3}^{d+2}\F_2[X_1^{\pm}, \ldots, X_{d+2}^{\pm}] \cdot (X_i - 1)$. So by Theorem~\ref{thm:main}, Submonoid Membership is decidable in $\widetilde{G}$.
\end{proof}

Our proof of Theorem~\ref{thm:main} is divided into two parts.
First, in Section~\ref{sec:SMtoSunit}, we reduce Submonoid Membership to solving \emph{S-unit equations} over modules. 
Then, in Section~\ref{sec:Sunit}, we prove the following result for S-unit equations over modules.

\begin{restatable}{thrm}{thmSunit}\label{thm:Sunit}
    Let $p$ be a prime number, $\oX = (X_1, \ldots, X_n)$ be a tuple of variables, $M$ be a finitely presented $\F_p[\oX^{\pm}]$-module, and let $c_0, c_1, \ldots, c_m \in M$.
    Then, the set of solutions $(\bz_1, \ldots, \bz_m) \in \left(\Z^{n}\right)^{m}$ for the equation
    \begin{equation}\label{eq:Sunitori}
        \oX^{\bz_1} \cdot c_1 + \cdots + \oX^{\bz_m} \cdot c_m = c_0
    \end{equation}
    is an effectively $p$-automatic subset of $\left(\Z^{n}\right)^{m} \cong \Z^{nm}$.
\end{restatable}

During the reduction in Section~\ref{sec:SMtoSunit}, we go through an intermediate step involving the \emph{Knapsack Problem} in groups. In particular, we obtain the following intermediate result that is of independent interest:

\begin{restatable}{thrm}{thmKP}\label{thm:KP}
    Let $p$ be a prime number, $\mY$ be a finitely presented $\F_p[\oX^{\pm}]$-module and $G = \mY \rtimes \Z^n$. 
    Let $h_1, \ldots, h_m, g$ be elements of $G$, then the solution set $(n_1, \ldots, n_m) \in \Z^m$ of
    \begin{equation}\label{eq:KP}
        h_1^{n_1} h_2^{n_2} \cdots h_m^{n_m} = g
    \end{equation}
    is an effectively $p$-automatic subset of $\Z^m$.
\end{restatable}

\section{From Submonoid Membership to S-unit equation}\label{sec:SMtoSunit}

In this section we reduce the problem of Submonoid Membership in $\mY \rtimes \Z^n$ to solving S-unit equations over $\F_p[\oX^{\pm}]$-modules.
The reduction is done in three steps. In Step 1, we reduce Submonoid Membership in $\mY \rtimes \Z^n$ to deciding membership in a product of conjugate groups (Proposition~\ref{prop:SMtoGP}). In Step 2, we reduce membership in a product of conjugate groups to solving the Knapsack Problem in another semidirect product $\mY' \rtimes \Z^{n'}$ (Proposition~\ref{prop:GPtoKP}). In Step 3, we show that the solution set of the Knapsack Problem is effectively $p$-automatic (Theorem~\ref{thm:KP}), assuming a theorem on S-unit equations (Theorem~\ref{thm:Sunit}).

Recall that the \emph{rank} of an abelian group $A$ is defined as smallest cardinality of a generating set for $A$.
If $A \leq \Z^n$, then $A$ is a lattice in $\Z^n$, the rank of $A$ is the dimension of this lattice.

\paragraph*{Step 1: from Submonoid Membership to Group Product Membership}

Given $q \in G$ and a subgroup $H \leq G$, define the conjugate subgroup $\leftindex^{q}{H} \coloneqq \{qhq^{-1} \mid h \in H\}$.
For $q_1, q_2, \ldots, q_k \in G$, denote the set
\begin{equation}\label{eq:prodgroup}
\leftindex^{q_1}{H} \cdot \leftindex^{q_2}{H} \cdot \ldots \cdot \leftindex^{q_{k}}{H} \coloneqq \{h_1 h_2 \cdots h_k \mid h_1 \in \leftindex^{q_1}{H}, \ldots, h_k \in \leftindex^{q_{k}}{H}\}.
\end{equation}
Note that $\leftindex^{q}{H}$ is always a group, but the product $\leftindex^{q_1}{H} \cdot \leftindex^{q_2}{H} \cdot \ldots \cdot \leftindex^{q_{k}}{H}$ is in general not a group.
Our first step is the following reduction from Submonoid Membership to the membership problem in sets of the form~\eqref{eq:prodgroup}, originally due to Shafrir.
We give a proof here for the sake of completeness.
Recall that the projection $\pi \colon G = \mY \rtimes \Z^n \rightarrow \Z^n$ is surjective and its kernel is $\mY$.
A group $A$ is called \emph{torsion} if for all $a \in A$, there exists $t \geq 1$ such that $a^t$ is the neutral element.
In particular, an $\F_p[\oX^{\pm}]$-module $\mY$ considered as an abelian group is torsion, because $p y = 0$ for all $y \in \mY$.

\begin{prop}[{\cite[Theorem~4.3.3]{potthast2020submonoid}}]\label{prop:SMtoGP}
    Let G be a finitely generated group and $\pi \colon G \rightarrow \Z^n$ be a surjective homomorphism with $\ker(\pi)$ being a torsion group. Submonoid Membership in $G$ can be reduced to the following membership problem:
    
    \textbf{Input:} elements $g, q_1, \ldots, q_k \in G$ and finitely many generators of a subgroup $H \leq G$.
    
    \textbf{Question:} whether $g \in \leftindex^{q_1}{H} \cdot \leftindex^{q_2}{H} \cdot \ldots \cdot \leftindex^{q_{k}}{H}$?
\end{prop}

\begin{proof}
    For a subset $T \subseteq G$, denote by $\gen{T}$ the submonoid generated by $T$: this is the set of elements of the form $g_1 g_2 \dots g_s$ where $s \geq 0$ and $g_i \in T$ for $1 \leq i \leq s$.
    
    Let $g_1, g_2, \ldots, g_m, g \in G$, and suppose we want to decide whether $g \in \gen{g_1, g_2, \ldots, g_m}$.
    Consider the elements $\pi(g_1), \ldots, \pi(g_m), \pi(g) \in \Z^n \subset \R^n$, let $\mC \subseteq \R^n$ be the $\R_{\geq 0}$-cone generated by $\pi(g_1), \pi(g_2), \ldots, \pi(g_m)$, and let $\mC_0$ be the maximal $\R$-subspace of $\mC$.
    In particular, there exists $v \in \R^n$, such that $v^{\top} \cdot w = 0$ for all $w \in \mC_0$, and $v^{\top} \cdot w > 0$ for all $w \in \mC \setminus \mC_0$ (if $\mC_0 = \mC$ then we can take $v = \bzer$).
    Without loss of generality suppose $\pi(g_1) , \ldots, \pi(g_{\ell}) \in \mC_0$ and $\pi(g_{\ell + 1}) , \ldots, \pi(g_{m}) \in \mC \setminus \mC_0$, for some $1 \leq \ell \leq m$.

    If $g \in \gen{g_1, g_2, \ldots, g_m}$, then we can write $g$ as a word $g_{i_1} g_{i_2} \cdots g_{i_s}$ over $\{g_1, \ldots, g_m\}$. 
    In this case, $\pi(g) = \pi(g_{i_1}) + \pi(g_{i_2}) + \cdots + \pi(g_{i_s})$, so we must have $v^{\top} \cdot \pi(g) \geq 0$, and the elements $g_{\ell + 1}, \ldots, g_m$ can only be used a bounded number of times (in fact, bounded by $B \coloneqq v^{\top} \cdot \pi(g)/ \min\left\{v^{\top} \cdot \pi(g_{\ell + 1}), \ldots, v^{\top} \cdot \pi(g_{m})\right\}$).
    Therefore, it suffices to check, for each possible choice of $g_{j_1}, \ldots, g_{j_b} \in \{g_{\ell + 1}, \ldots, g_m\}, \; b \leq B$, whether $g$ can be written as $h_1 g_{j_1} h_2 g_{j_2} h_3 \cdots g_{j_b} h_{b+1}$ for $h_1, \ldots, h_{b+1} \in \gen{g_1, g_2, \ldots, g_{\ell}}$.
    
    Let $H \coloneqq \gen{g_1, g_2, \ldots, g_{\ell}}$, we claim that the monoid $H$ is actually a group. 
    It is easy to see that the $\R_{\geq 0}$-cone generated by $\pi(g_1) , \ldots, \pi(g_{\ell})$ is the linear space $\mC_0$.
    Therefore we can find positive integers $n_1, \ldots, n_{\ell} > 0$ such that $n_1 \pi(g_1) + \cdots + n_{\ell} \pi(g_{\ell}) = \bzer$. Thus, $\pi(g_1^{n_1} g_2^{n_2} \cdots g_{\ell}^{n_{\ell}}) = \bzer$, so $g_1^{n_1} g_2^{n_2} \cdots g_{\ell}^{n_{\ell}} \in \ker(\pi)$.
    But $\ker(\pi)$ is torsion, so $\left( g_1^{n_1} g_2^{n_2} \cdots g_{\ell}^{n_{\ell}} \right)^t$ is the neutral element for some $t \geq 1$.
    This yields $g_1^{-1} = g_1^{n_1 - 1} g_2^{n_2} \cdots g_{\ell}^{n_{\ell}} \left( g_1^{n_1} g_2^{n_2} \cdots g_{\ell}^{n_{\ell}} \right)^{t-1} \in H$.
    Similarly, $g_2^{-1}, \ldots, g_{\ell}^{-1} \in H$, so the monoid $H$ is actually a group.

    By the discussion above, it suffices to check for finitely many instances of $g_{j_1}, \ldots, g_{j_b}$, whether $g \in H g_{j_1} H g_{j_2} H \cdots g_{j_b} H$.
    But $H g_{j_1} H g_{j_2} H \cdots g_{j_b} H$ is equal to
    \[
    H (g_{j_1} H g_{j_1}^{-1}) (g_{j_1} g_{j_2} H g_{j_2}^{-1} g_{j_1}^{-1}) \cdots (g_{j_1} g_{j_2} \cdots g_{j_b} H g_{j_b}^{-1} \cdots g_{j_2}^{-1} g_{j_1}^{-1}) g_{j_1} g_{j_2} \cdots g_{j_b}.
    \]
    So $g \in H g_{j_1} H g_{j_2} H \cdots g_{j_b} H$ is equivalent to
    \[
    g g_{j_b}^{-1} \cdots g_{j_2}^{-1} g_{j_1}^{-1} \in H \cdot \leftindex^{g_{j_1}}{H} \cdot \leftindex^{g_{j_1} g_{j_2}}{H} \cdot \ldots \cdot \leftindex^{g_{j_1} g_{j_2} \cdots g_{j_b}}{H},
    \]
    which is a product of conjugate groups.
    This finishes the reduction from Submonoid Membership to membership in products of conjugate groups.
\end{proof}

In particular, for the case of $G =  \mY \rtimes \Z^n$, Proposition~\ref{prop:SMtoGP} is effective, meaning one can explicitly compute the elements $q_1, \ldots, q_k \in G$ and the generators of $H$.

\begin{exmpl}[label=exa:cont]\label{expl:running}
Let $\mY = \F_2[X_1^{\pm}, X_2^{\pm}]$ considered as an $\F_2[X_1^{\pm}, X_2^{\pm}]$-module, so $\mY \rtimes \Z^2$ is the 2-dimensional lamplighter group $(\Z/2\Z) \wr \Z^2$.
Let $g_1 = \big(1 + X_2, \, (2, 0)\big), g_2 = \big(1, \, (- 2, 0)\big), g_3 = \big(1 + X_1, \, (0, 1)\big)$ be the generators of a submonoid, and let $g = \big(X_2^2, \, (4, 2)\big)$.

We have $\pi(g_1) = (2, 0), \; \pi(g_2) = (-2, 0), \; \pi(g_3) = (0, 1)$, so the $\R_{\geq 0}$-cone they generate is $\mC = \{(x, y) \in \R^2 \mid y \geq 0\}$.
Its maximal subspace is $\mC_0 = \{(x, 0) \mid x \in \R\}$.
We have $\pi(g_1), \pi(g_2) \in \mC_0$ and $\pi(g_3) \in \mC \setminus \mC_0$.

Suppose $g = g_{i_1} g_{i_2} \cdots g_{i_s}$ with $i_1, i_2, \ldots, i_s \in \{1, 2, 3\}$.
Let $v = (0, 1)$.
Since $v^{\top} \cdot \pi(g) = v^{\top} \cdot (4, 2) = 2$, and $v^{\top} \cdot \pi(g_1) = v^{\top} \cdot \pi(g_2) = 0, \; v^{\top} \cdot \pi(g_3) = 1$, the element $g_3$ must appear exactly twice in the product $g_{i_1} g_{i_2} \cdots g_{i_s}$.
Therefore, $g$ must be of the form $w_1 g_3 w_2 g_3 w_3$, where $w_1, w_2, w_3$ are words over $\{g_1, g_2\}$.
Let $H$ be the monoid generated by $g_1, g_2$: this is actually a group because $(g_1 g_2)^2 = (0, \bzer)$.
The above discussion shows $g \in \gen{g_1, g_2, g_3}$ if and only if $g \in H g_3 H g_3 H$: this is equivalent to $g g_3^{-2} \in H \cdot \leftindex^{g_3}{H} \cdot \leftindex^{g_3^{2}}{H}$.
\hfill $\blacksquare$
\end{exmpl}

\begin{rmk}
If $G = (\Z/2\Z) \wr \Z^2$ and $\pi(H)$ is of rank one (i.e.\ $\pi(H) \cong \Z$) as in Example~\ref{expl:running}, then membership problem in the product $\leftindex^{q_1}{H} \cdot \leftindex^{q_2}{H} \cdot \ldots \cdot \leftindex^{q_{k}}{H}$ can be reduced to Rational Subset Membership in the 1-dimensional lamplighter group $(\Z/2\Z) \wr \Z$, which is known to be decidable~\cite{lohrey2015rational}. In fact, this is the key idea in Shafrir's solution~\cite{potthast2020submonoid} for Submonoid Membership in 2-dimensional lamplighter groups.
However, when the dimension $n$ is greater or equal to three, the rank of $\pi(H)$ may be greater or equal to two: in this case, the membership problem in $\leftindex^{q_1}{H} \cdot \leftindex^{q_2}{H} \cdot \ldots \cdot \leftindex^{q_{k}}{H}$ has been open and will be solved in the rest of this paper.
Notably, we do \emph{not} want to reduce it to Rational Subset Membership in lamplighter groups of dimension $\geq 2$, as it is undecidable~\cite{lohrey2011tilings}.
\end{rmk}

\paragraph*{Step 2: from Group Product Membership to Knapsack Problem}
The goal of this step is to reduce membership problem in the product $\leftindex^{q_1}{H} \cdot \leftindex^{q_2}{H} \cdot \ldots \cdot \leftindex^{q_{k}}{H}$ to a version of the \emph{Knapsack Problem} in groups.
Namely, we will prove:
\begin{prop}\label{prop:GPtoKP}
    Let $G = \mY \rtimes \Z^n$ where $\mY$ is a finitely presented $\F_p[X_1^{\pm}, \ldots, X_{n}^{\pm}]$-module.
    Let $g, q_1, \ldots, q_k \in G$ and let $H \leq G$ be a finitely generated subgroup.
    Then, one can effectively find a group $G' = \mY' \rtimes \Z^{n'}$, where $\mY'$ is a finitely presented $\F_p[X_1^{\pm}, \ldots, X_{n'}^{\pm}]$-module, as well as elements $g', h_1, \ldots, h_m \in G'$, such that
    \[
    g \in \leftindex^{q_1}{H} \cdot \leftindex^{q_2}{H} \cdot \ldots \cdot \leftindex^{q_{k}}{H}
    \]
    if and only if there exist $z_1, \ldots, z_m \in \Z$ satisfying
    \begin{equation}\label{eq:KP2}
        h_1^{z_1} h_2^{z_2} \cdots h_m^{z_m} = g'.
    \end{equation}
\end{prop}

We now proceed to prove Proposition~\ref{prop:GPtoKP}.
Recall that $\mY$ can be identified with the subgroup $\mY \times \{\bzer\} = \ker(\pi)$ of $G = \mY \rtimes \Z^n$. 
First, we give the following classic result on the structure of subgroups $H \leq G$.
Note that $\pi(H) \leq \pi(G) = \Z^n$, let $d \leq n$ be the rank of $\pi(H)$.

\begin{lem}[{\cite[Theorem~3.3]{baumslag1994algorithmic} or \cite[Lemma~2]{romanovskii1974some}}]\label{lem:structmeta}
    Let $G = \mY \rtimes \Z^n$ and $H \leq G$ be a finitely generated subgroup. Let $\ba_1, \ldots, \ba_{d} \in \Z^n$ be a $\Z$-basis of $\pi(H)$, and choose elements $h_1, \ldots, h_d \in H$ so that $\pi(h_i) = \ba_i, i = 1, \ldots, d$.
    Then, the subgroup $\mY \cap H$ is a finitely generated $\F_p[\oX^{\pi(H)}]$-module, whose finite presentation can be effectively computed. Furthermore, $H$ is generated by $h_1, \ldots, h_d$ and the set $\mY \cap H$.
\end{lem}

\renewcommand\thmcontinues[1]{continued}
\begin{exmpl}[continues=exa:cont]
As in Example~\ref{expl:running}, let $\mY = \F_2[X_1^{\pm}, X_2^{\pm}]$.
Consider the subgroup $H \leq \mY \rtimes \Z^2$ generated by $g_1 = \big(1 + X_2, \, (2, 0)\big)$ and $g_2 = \big(1, \, (- 2, 0)\big)$.
Then $\pi(H)$ is the subgroup of $\Z^2$ generated by $\ba_1 \coloneqq (2, 0)$, it has rank $d = 1$. 
We can choose $h_1 \coloneqq g_1$ in Lemma~\ref{lem:structmeta}.

The subgroup $\mY \cap H$ is an $\F_p[\oX^{\pi(H)}] = \F_p[X_1^2, X_1^{-2}]$-module. We have $g_1 g_2 = \big(1 + X_2 + X_1^2, \bzer \big)$.
We claim that $1 + X_2 + X_1^2$ is the generator of the $\F_p[X_1^2, X_1^{-2}]$-module $\mY \cap H$:
that is, $\mY \cap H = \{F \cdot (1 + X_2 + X_1^2) \mid F \in \F_p[X_1^2, X_1^{-2}]\}$. 
Indeed, recall the formula $(f, \ba)(y, \bzer)(f, \ba)^{-1} = (\oX^{\ba} \cdot y, \bzer)$.
For each $m \in \Z$, we have $g_1^m (g_1 g_2) g_1^{-m} = g_1^m \big(1 + X_2 + X_1^2, \bzer\big) g_1^{-m} = \big(X_1^{2m} \cdot (1 + X_2 + X_1^2), \bzer\big)$.
So $X_1^{2m} \cdot (1 + X_2 + X_1^2) \in \mY \cap H$ for all $m \in \Z$.
Consequently, $F \cdot (1 + X_2 + X_1^2) \in \mY \cap H$ for all $F \in \F_p[X_1^2, X_1^{-2}]$.
It is not difficult to show that $\{F \cdot (1 + X_2 + X_1^2) \mid F \in \F_p[X_1^2, X_1^{-2}]\}$ is indeed all the elements in $\mY \cap H$, but we will omit the details.
Note that $\mY \cap H$ is finitely generated as an $\F_p[X_1^2, X_1^{-2}]$-module but not finitely generated \emph{as an abelian group}.

Finally, for any element $h = (f, 2z) \in H$, we have $\pi(h_1^{-z} h) = \bzer$, so $(y, \bzer) \coloneqq h_1^{-z} h \in \mY \cap H$. 
Consequently, $h$ can actually be written in a ``canonical form'' $h_1^z \, (y, \bzer)$, where $z \in \Z$ and $y \in \mY \cap H$.
Such canonical forms will be extensively used below.
\hfill $\blacksquare$
\end{exmpl}

Define
\[
\mY_i \coloneqq \mY \cap \left(\leftindex^{q_i}{H}\right), \quad i = 1, \ldots, k.
\]
Note that for any $q \in G$, we have $\pi(\leftindex^{q}{H}) = \pi(H)$ because $\Z^n$ is abelian. Therefore $\pi(\leftindex^{q_1}{H}) = \cdots = \pi(\leftindex^{q_k}{H}) = \pi(H)$.
So by Lemma~\ref{lem:structmeta}, all the $\mY_i$'s are finitely generated $\F_p[\oX^{\pi(H)}]$-modules.
Let $\ba_1, \ldots, \ba_{d} \in \Z^n$ be a $\Z$-basis of $\pi(H)$.
For each $1 \leq i \leq k$, choose elements $h_{i,1}, \ldots, h_{i,d} \in \leftindex^{q_i}{H}$ so that $\pi(h_{i,1}) = \ba_1, \ldots, \pi(h_{i,d}) = \ba_d$.
Then, $\leftindex^{q_i}{H}$ is generated by $h_{i,1}, \ldots, h_{i,d}$ and $\mY_i$.

Fix any $i$. Since $\ba_1, \ldots, \ba_{d} \in \Z^n$ form a $\Z$-basis for $\pi(\leftindex^{q_i}{H})$, for each element $h \in \leftindex^{q_i}{H}$ we can find integers $z_1, \ldots, z_d$ such that 
\[
\pi(h) = z_1 \ba_1 + z_2 \ba_2 + \cdots + z_d \ba_d = z_1 \pi(h_{i,1}) + z_2 \pi(h_{i,2}) + \cdots + z_d \pi(h_{i,d}).
\]
This shows $\pi(h_{i, d}^{- z_{i, d}} \cdots h_{i, 2}^{- z_{i, 2}} h_{i, 1}^{- z_{i, 1}} h) = \bzer$, so $h_{i, d}^{- z_{i, d}} \cdots h_{i, 2}^{- z_{i, 2}} h_{i, 1}^{- z_{i, 1}} h \in \ker(\pi) \cap \leftindex^{q_i}{H} = \mY_i$.
Consequently, $h$ is equal to $h_{i,1}^{z_{i, 1}} h_{i,2}^{z_{i, 2}} \cdots h_{i,d}^{z_{i, d}} (y_i, \bzer)$ for some $y_i \in \mY_i$.

The form of this product of powers gives us the motivation to reduce membership in group products to the Knapsack Problem. 
Denote
\[
\tY \coloneqq \mY_1 + \cdots + \mY_k.
\]
Again we identify $\tY$ with the subgroup $\tY \times \{\bzer\}$ of $\mY \rtimes \Z^n$. We show the following:

\begin{lem}\label{lem:SMtoKP}
    We have $g \in \leftindex^{q_1}{H} \cdot \leftindex^{q_2}{H} \cdot \ldots \cdot \leftindex^{q_{k}}{H}$, if and only if there exist integers $z_{1, 1}, \ldots, z_{1, d}$, $\ldots$, $z_{k, 1}, \ldots, z_{k, d}$, such that
    \begin{equation}\label{eq:SMtoKP}
        \left(h_{1,1}^{z_{1, 1}} h_{1,2}^{z_{1, 2}} \cdots h_{1,d}^{z_{1, d}}\right) \cdots \left(h_{k,1}^{z_{k, 1}} h_{k,2}^{z_{k, 2}} \cdots h_{k,d}^{z_{k, d}}\right) \cdot g^{-1} \in \tY.
    \end{equation}
\end{lem}
\begin{proof}
    For the ``only if'' part, suppose $g \in \leftindex^{q_1}{H} \cdot \leftindex^{q_2}{H} \cdot \ldots \cdot \leftindex^{q_{k}}{H}$, and write $g = h_1 h_2 \cdots h_k$ with $h_i \in \leftindex^{q_i}{H}$ for each $1 \leq i \leq k$.
    As shown above, we can write
    \[
    h_i = h_{i,1}^{z_{i, 1}} h_{i,2}^{z_{i, 2}} \cdots h_{i,d}^{z_{i, d}} (y_i, \bzer)
    \]
    for some $y_i \in \mY_i, \, i = 1, \ldots, k$.
    Denote $y_i' \coloneqq \oX^{\sum_{j = 1}^d z_{1, j} \ba_j + \sum_{j = 1}^d z_{2, j} \ba_j + \cdots + \sum_{j = 1}^d z_{i, j} \ba_j} \cdot y_i$, then $y_i' \in \mY_i$.
    Recall the formula $(f, \ba)(y, \bzer)(f, \ba)^{-1} = (\oX^{\ba} \cdot y, \bzer)$, that is $(f, \ba)(y_i, \bzer) = (\oX^{\ba} \cdot y_i, \bzer) (f, \ba)$. Therefore,
    \begin{align*}
        g & = h_1 h_2 \cdots h_k 
        = \left(h_{1,1}^{z_{1, 1}} h_{1,2}^{z_{1, 2}} \cdots h_{1,d}^{z_{1, d}}\right) (y_1, \bzer) \cdots \left(h_{k,1}^{z_{k, 1}} h_{k,2}^{z_{k, 2}} \cdots h_{k,d}^{z_{k, d}}\right) (y_k, \bzer) \\
        & = (y_k', \bzer) \cdots (y_2', \bzer) (y_1', \bzer) \left(h_{1,1}^{z_{1, 1}} h_{1,2}^{z_{1, 2}} \cdots h_{1,d}^{z_{1, d}}\right) \cdots \left(h_{k,1}^{z_{k, 1}} h_{k,2}^{z_{k, 2}} \cdots h_{k,d}^{z_{k, d}}\right)
    \end{align*}
    Therefore 
    \[
    \left(h_{1,1}^{z_{1, 1}} h_{1,2}^{z_{1, 2}} \cdots h_{1,d}^{z_{1, d}}\right) \cdots \left(h_{k,1}^{z_{k, 1}} h_{k,2}^{z_{k, 2}} \cdots h_{k,d}^{z_{k, d}}\right) \cdot g^{-1} = (- y_k' - \cdots - y_2' - y_1', \bzer) \in \mY_1 + \cdots + \mY_k = \tY.
    \]

    For the ``if'' part, let $z_{1, 1}, \ldots, z_{1, d}$, $\ldots$, $z_{k, 1}, \ldots, z_{k, d}$ be integers that satisfy Equation~\eqref{eq:SMtoKP}. Since $\tY \coloneqq \mY_1 + \cdots + \mY_k$, there exist $y_1 \in \mY_1, \ldots, y_k \in \mY_k$, such that
    \[
    \left(h_{1,1}^{z_{1, 1}} h_{1,2}^{z_{1, 2}} \cdots h_{1,d}^{z_{1, d}}\right) \cdots \left(h_{k,1}^{z_{k, 1}} h_{k,2}^{z_{k, 2}} \cdots h_{k,d}^{z_{k, d}}\right) \cdot g^{-1} = (y_1 + y_2 + \cdots + y_k, \bzer).
    \]
    Denote $y_i' \coloneqq - \oX^{- \sum_{j = 1}^d z_{1, j} \ba_j - \sum_{j = 1}^d z_{2, j} \ba_j - \cdots - \sum_{j = 1}^d z_{i, j} \ba_j} \cdot y_i $, then $y_i' \in \mY_i$. 
    We have
    \begin{align*}
        g & = (- y_1 - y_2 - \cdots - y_k, \bzer) \left(h_{1,1}^{z_{1, 1}} h_{1,2}^{z_{1, 2}} \cdots h_{1,d}^{z_{1, d}}\right) \cdots \left(h_{k,1}^{z_{k, 1}} h_{k,2}^{z_{k, 2}} \cdots h_{k,d}^{z_{k, d}}\right) \\
        & = (- y_k, \bzer) \cdots (- y_2, \bzer) (- y_1, \bzer) \left(h_{1,1}^{z_{1, 1}} h_{1,2}^{z_{1, 2}} \cdots h_{1,d}^{z_{1, d}}\right) \cdots \left(h_{k,1}^{z_{k, 1}} h_{k,2}^{z_{k, 2}} \cdots h_{k,d}^{z_{k, d}}\right) \\
        & = \left(h_{1,1}^{z_{1, 1}} h_{1,2}^{z_{1, 2}} \cdots h_{1,d}^{z_{1, d}}\right) (y_1', \bzer) \cdots \left(h_{k,1}^{z_{k, 1}} h_{k,2}^{z_{k, 2}} \cdots h_{k,d}^{z_{k, d}}\right) (y_k', \bzer).
    \end{align*}
    For each $i$, we have $(y_i', \bzer) \in \mY_i \leq \leftindex^{q_i}{H}$, so $\left(h_{i,1}^{z_{i, 1}} h_{i,2}^{z_{i, 2}} \cdots h_{i,d}^{z_{i, d}}\right) (y_i', \bzer) \in \leftindex^{q_i}{H}$. Therefore $g \in \leftindex^{q_1}{H} \cdot \leftindex^{q_2}{H} \cdot \ldots \cdot \leftindex^{q_{k}}{H}$.
\end{proof}

Proposition~\ref{prop:GPtoKP} then follows from Lemma~\ref{lem:SMtoKP}:
\begin{proof}[Proof of Proposition~\ref{prop:GPtoKP}]
By Lemma~\ref{lem:SMtoKP}, we have $g \in \leftindex^{q_1}{H} \cdot \leftindex^{q_2}{H} \cdot \ldots \cdot \leftindex^{q_{k}}{H}$ if and only if Equation~\eqref{eq:SMtoKP} has integer solutions.
Write $h_{i, j}$ as $(f_{i, j}, \ba_j)$ for $i = 1, \ldots, k; \, j = 1, \ldots, d$, and write $g$ as $(f_0, \ba_0)$.
Let $\mF$ denote the $\F_p[\oX^{\pi(H)}]$-submodule of $\mY$ generated by $f_0$ and $f_{i, j}, \, i = 1, \ldots, k; \, j = 1, \ldots, d$.
Notably, $\mF$ contains $\mY_i, \, i = 1, \ldots, k$.
A finite presentation of $\mF$ can be computed (see Lemma~\ref{lem:sgmod}).
The group generated by $g$ and all the $h_{i, j}$'s is contained in the subgroup $\mF \rtimes \pi(H)$ of $\mY \rtimes \Z^n$.
Hence, Equation~\eqref{eq:SMtoKP} is equivalent to the following equation in $(\mF/\tY) \rtimes \pi(H)$:
\begin{equation}\label{eq:blockKP}
        \left(h_{1,1}^{z_{1, 1}} h_{1,2}^{z_{1, 2}} \cdots h_{1,d}^{z_{1, d}}\right) \cdots \left(h_{k,1}^{z_{k, 1}} h_{k,2}^{z_{k, 2}} \cdots h_{k,d}^{z_{k, d}}\right) = g.
\end{equation}
Choose the new variables $X_1' \coloneqq \oX^{\ba_1}, \ldots, X_d' \coloneqq \oX^{\ba_d}$.
Then the ring $\F_p[\oX^{\pi(H)}]$ can be written as the polynomial ring $\F_p[{X_1'}^{\pm}, \ldots, {X_d'}^{\pm}]$, and $\mY' \coloneqq \mF/\tY$ is now a finitely presented $\F_p[{X_1'}^{\pm}, \ldots, {X_d'}^{\pm}]$-module.
We conclude by letting $G' \coloneqq \mY' \rtimes \Z^d = (\mF/\tY) \rtimes \pi(H)$, and observing that Equation~\eqref{eq:blockKP} is a Knapsack equation of the form~\eqref{eq:KP2}.
\end{proof}

\begin{rmk}\label{rmk:noteasy}
Note that even if we only considered the lamplighter group $G = (\Z/p\Z) \wr \Z^n$ by supposing $\mY = \F_p[\oX]$, we would inevitably introduce the quotient by $\tY$ during our reduction from Submonoid Membership to the Knapsack Problem. Therefore, it is reasonable to directly consider Submonoid Membership in the more general case of semidirect product $\mY \rtimes \Z^n$ with \emph{finitely presented} $\mY$, instead of only the special case of lamplighter groups $(\Z/p\Z) \wr \Z^n$. In other words, it does not appear that the special case of lamplighter groups admits a more direct solution.
\end{rmk}


\paragraph*{Step 3: from Knapsack Problem to S-unit equation}

The goal of this step is to solve the Knapsack Problem in groups of the form $\mY \rtimes \Z^n$. Namely, assuming a theorem on S-unit equations (Theorem~\ref{thm:Sunit}), we will prove the following result:

\thmKP*

Before we proceed to the formal proof, let us show another example to illustrate the ideas. 

\begin{exmpl}\label{expl:KPtoSunit}
    Let $p = 2$, $n = 2$.
    Let $\mY = \F_2[X_1^{\pm}, X_2^{\pm}]/(X_1+X_2+1)$, that is, the quotient of $\F_2[X_1^{\pm}, X_2^{\pm}]$ by the submodule $\F_2[X_1^{\pm}, X_2^{\pm}] \cdot (X_1+X_2+1)$.    
    Let $h_1 = \big(0, (1, 0)\big), h_2 = \big(1 - X_2, (0, 1)\big), h_3 = \big(0, (1, 0)\big), h_4 = \big(0, (0, 1)\big), g = \big(1, (0, 0)\big)$ be elements of $\mY \rtimes \Z^2$. We will to show that the solution set of the equation
    \begin{equation}\label{eq:exampleKP}
        h_1^{n_1} h_2^{n_2} h_3^{n_3} h_4^{n_4} = g
    \end{equation}
    is 2-automatic.

    By direct computation, we have $(f, \ba)^z = \big((1 + \oX^{\ba} + \cdots + \oX^{(z-1) \ba}) \cdot f, z \ba\big)$ for $z > 0$ and $(f, \ba)^z = \big((- \oX^{z \ba} - \oX^{(z+1) \ba} - \cdots - \oX^{-\ba}) \cdot f, z \ba\big)$ for $z < 0$.
    In either case (and for $z = 0$), we have
    \[
    (f, \ba)^z = \left( \frac{1 - \oX^{z \ba}}{1 - \oX^{\ba}} \cdot f, \; z \ba \right).
    \]
    Hence,
    \begin{align*}
        h_1^{n_1} h_2^{n_2} h_3^{n_3} h_4^{n_4} & = \big(0, (1, 0)\big)^{n_1} \big(1 - X_2, (0, 1)\big)^{n_2} \big(0, (1, 0)\big)^{n_3} \big(0, (0, 1)\big)^{n_4} \\
        & = \big(0, (n_1, 0)\big) \left(\frac{1 - X_2^{n_2}}{1 - X_2} \cdot (1 - X_2), (0, n_2) \right) \big(0, (n_3, 0)\big) \big(0, (0, n_4)\big) \\
        & = \big(X_1^{n_1} (1 - X_2^{n_2}), (n_1 + n_3, n_2 + n_4) \big)
    \end{align*}
    Since $g = \big(1, (0, 0)\big)$, we have $h_1^{n_1} h_2^{n_2} h_3^{n_3} h_4^{n_4} = g$ if and only if
    \[
    X_1^{n_1} (1 - X_2^{n_2}) = 1, \; n_1 + n_3 = 0, \; n_2 + n_4 = 0.
    \]
    We now show that the solution set $(n_1, n_2, n_3, n_4) \in \Z^4$ is 2-automatic. Since $n_3 = - n_1$ and $n_4 = - n_2$, it suffices to show that the solution set $(n_1, n_2) \in \Z^2$ of $X_1^{n_1} (1 - X_2^{n_2}) = 1$ is 2-automatic.

    The equation $X_1^{n_1} (1 - X_2^{n_2}) = 1$ can be rewritten as 
    \begin{equation}\label{eq:exampleSunit}
        X_1^{- n_1} + X_2^{n_2} = 1
    \end{equation}
    Therefore, we have reduced the \emph{Knapsack equation}~\eqref{eq:exampleKP} to the \emph{S-unit equation}~\eqref{eq:exampleSunit} over the module $\F_2[X_1^{\pm}, X_2^{\pm}]/(X_1+X_2+1)$.
    The 2-automaticity of its solution set can already be deduced from Theorem~\ref{thm:Sunit} (see below).
    However, we now prove ``by hand'' that the solution set of the S-unit Equation~\eqref{eq:exampleSunit} is 2-automatic, in order to give an idea of how 2-automaticity can arise from such equations.

    Since we work in the module $\F_2[X_1^{\pm}, X_2^{\pm}]/(X_1+X_2+1)$, we have $X_2 = - (X_1 + 1) = X_1 + 1$ (we are in characteristic 2).
    Therefore, Equation~\eqref{eq:exampleSunit} is equivalent to $(X_1 + 1)^{n_2} = X_1^{- n_1} + 1$.
    We claim that its solution set is $(n_1, n_2) \in \{(-2^k, 2^k) \mid k \in \N\}$.
    
    Indeed, since $X_1^{- n_1} + 1 \in \F_2[X_1^{\pm}]$, we must have $n_2 \geq 0$ so that $(X_1 + 1)^{n_2}$ is also in $\F_2[X_1^{\pm}]$. Considering the degree of both sides we have $n_1 = -n_2$. Now, in order for $(X_1 + 1)^{n_2} = X_1^{n_2} + 1$ to hold, $n_2$ must be a power of 2. 
    Indeed, $n_2 = 0$ is not a solution, so suppose $n_2 \geq 1$ and write $n_2 = 2^k + a$ with $0 \leq a < 2^k$.
    By ``freshman's dream'', we have
    \[
    (X_1 + 1)^{n_2} = (X_1 + 1)^{2^k} (X_1 + 1)^a = (X_1^{2^k} + 1) (X_1 + 1)^a = X_1^{2^k} (X_1 + 1)^a + (X_1 + 1)^a.
    \]
    Every monomial appearing in $X_1^{2^k} (X_1 + 1)^a$ has degree larger than every monomial appearing $(X_1 + 1)^a$. But $(X_1 + 1)^{n_2} = X_1^{n_2} + 1$ has only two monomials.
    So both $X_1^{2^k} (X_1 + 1)^a$ and $(X_1 + 1)^a$ must be monomials, meaning $a = 0$.
    Therefore $n_2 = 2^k$, and $n_1 = - n_2 = - 2^k$.

    In conclusion, the solution set of Equation~\eqref{eq:exampleKP} is $(n_1, n_2, n_3, n_4) \in \{(-2^k, 2^k, 2^k, - 2^k) \mid k \in \N\}$, which is 2-automatic. 
    Note that this set is not semilinear, so the situation in semidirect products is more complicated than in the wreath product of two abelian groups~\cite{DBLP:conf/stacs/GanardiKLZ18}.
    \hfill $\blacksquare$
\end{exmpl}

Example~\ref{expl:KPtoSunit} shows that there will be two parts in the proof of Theorem~\ref{thm:KP}: a (fairly elementary) first part reducing Knapsack Equations to S-unit equations, then a (rather non-trivial) second part showing that S-unit equations have $p$-automatic solution sets.
The second part is summarized as Theorem~\ref{thm:Sunit} below, which we will prove in Section~\ref{sec:Sunit}. In the rest of this section, we will focus on the first part of the proof, assuming Theorem~\ref{thm:Sunit} as a blackbox.

\thmSunit*

\begin{proof}[Proof of Theorem~\ref{thm:KP} (assuming Theorem~\ref{thm:Sunit})]
    Write $h_i = (f_i, \ba_i), \, i = 1, \ldots, m$, and $g = (f, \bb)$.
    
    \textbf{Case 1: none of the $\ba_i$ is $\bzer$} (in fact this case already suffices to solve the Equation~\eqref{eq:KP} arising from the previous step).
    We have
    \[
    h_i^{n_i} = (f_i, \ba_i)^{n_i} = \left( \frac{1 - \oX^{n_i \ba_i}}{1 - \oX^{\ba_i}} \cdot f_i, \; n_i \ba_i \right).
    \]
    Computing their product gives
    \[
    h_1^{n_1} h_2^{n_2} \cdots h_m^{n_m} = \left( \sum_{i = 1}^m \oX^{n_1 \ba_1 + \cdots + n_{i-1} \ba_{i-1}} \cdot \frac{1 - \oX^{n_i \ba_i}}{1 - \oX^{\ba_i}} \cdot f_i,\; n_1 \ba_1 + \cdots + n_m \ba_m \right).
    \]
    Therefore, writing $\bz_i \coloneqq n_1 \ba_1 + \cdots + n_{i} \ba_{i}$ for $i = 1, \ldots, m$, and $\bz_0 \coloneqq \bzer$, we have $h_1^{n_1} h_2^{n_2} \cdots h_m^{n_m} = g$ if and only if
    \begin{align}
        \sum_{i = 1}^m \oX^{\bz_{i-1}} \cdot \frac{1 - \oX^{\bz_i - \bz_{i-1}}}{1 - \oX^{\ba_i}} \cdot f_i & = f, \label{eq:sumraw} \\
        \text{and} \quad \bz_m & = \bb. \label{eq:sumlin}
    \end{align}    
    Write $\mY = \F_p[\oX^{\pm}]^d/N$ for some $N \subseteq \F_p[\oX^{\pm}]^d$.
    Considering $f_1, \ldots, f_m, f,$ as elements in $\F_p[\oX^{\pm}]^d$, Equation~\eqref{eq:sumraw} is equivalent to
    \[
    \sum_{i = 1}^m (\oX^{\bz_{i-1}} - \oX^{\bz_i}) \prod_{j \neq i} (1 - \oX^{\ba_j}) \cdot f_i - \prod_{i=1}^{m}(1 - \oX^{\ba_i}) \cdot f \in \prod_{i=1}^{m}(1 - \oX^{\ba_i}) \cdot N.
    \]
    Letting $c_0 \coloneqq - \prod_{j \neq 1} (1 - \oX^{\ba_j}) \cdot f_1 + \prod_{i=1}^{m}(1 - \oX^{\ba_i}) \cdot f$, $\; c_i \coloneqq \prod_{j \neq (i+1)} (1 - \oX^{\ba_j}) \cdot f_{i+1} - \prod_{j \neq i} (1 - \oX^{\ba_j}) \cdot f_i, \; i = 1, \ldots, m-1$ and $c_m \coloneqq - \prod_{j \neq m} (1 - \oX^{\ba_j}) \cdot f_m$, the equation above is equivalent to the following equation in $M \coloneqq \F_p[\oX^{\pm}]^d/\prod_{i=1}^{m}(1 - \oX^{\ba_i})N$:
    \begin{equation}\label{eq:sumfinal}
        \oX^{\bz_1} \cdot c_1 + \cdots + \oX^{\bz_m} \cdot c_m = c_0.
    \end{equation}
    Therefore, we have $h_1^{n_1} h_2^{n_2} \cdots h_m^{n_m} = g$ if and only if $\bz_i \coloneqq n_1 \ba_1 + \cdots + n_{i} \ba_{i}, i = 1, \ldots, m$ satisfy Equations~\eqref{eq:sumlin} and \eqref{eq:sumfinal}. By Theorem~\ref{thm:Sunit}, the set $\mZ$ of $(\bz_1, \ldots, \bz_m) \in \Z^{mn}$ satisfying \eqref{eq:sumlin} and \eqref{eq:sumfinal} is $p$-automatic.
    Define the map $\varphi \colon \Z^m \rightarrow \Z^{mn}, (n_1, \ldots, n_m) \mapsto (n_1 \ba_1, n_1 \ba_1 + n_2 \ba_2, \ldots, n_1 \ba_1 + \cdots + n_m \ba_m)$. Then the solution set of $h_1^{n_1} h_2^{n_2} \cdots h_m^{n_m} = g$ is $\varphi^{-1}(\mZ)$. Since $\mZ$ is $p$-automatic and $\varphi$ is linear, $\varphi^{-1}(\mZ)$ is also $p$-automatic (see Lemma~\ref{lem:pauto}).

    \textbf{Case 2: some of the $\ba_i$'s are $\bzer$.} 
    We reduce this case to the previous case. Let $i_1, \ldots, i_s$ be the indices where $\ba_i = \bzer$.
    We rewrite the equation $h_1^{n_1} h_2^{n_2} \cdots h_m^{n_m} = g$ as 
    \[
    \left(h_1^{n_1} \cdots h_{i_1 - 1}^{n_{i_1 - 1}}\right) h_{i_1}^{n_{i_1}} \left(h_{i_1 + 1}^{n_{i_1 + 1}} \cdots h_{i_2 - 1}^{n_{i_2 - 1}} \right) h_{i_2}^{n_{i_2}} \cdots \left(h_{i_s + 1}^{n_{i_s + 1}} \cdots h_{m}^{n_{m}} \right) = g.
    \]
    Since $\ba_{i_1} = \cdots = \ba_{i_s} = \bzer$ and $p \cdot f_{i_1} = \cdots = p \cdot f_{i_s} = 0$, we have $h_{i_1}^p = \cdots = h_{i_s}^p = (0, \bzer)$. Therefore, the powers $h_{i_1}^{n_{i_1}}, h_{i_2}^{n_{i_2}}, \ldots,$ and $h_{i_s}^{n_{i_s}}$ can only take on finitely many different values (at most $p$ each). It suffices to show that for each value $h_{i_1}^{n_{i_1}} = g_1, \ldots, h_{i_s}^{n_{i_s}} = g_s$, the solution set of
    \[
    \left(h_1^{n_1} \cdots h_{i_1 - 1}^{n_{i_1 - 1}}\right) g_1 \left(h_{i_1 + 1}^{n_{i_1 + 1}} \cdots h_{i_2 - 1}^{n_{i_2 - 1}} \right) g_2 \cdots \left(h_{i_s + 1}^{n_{i_s + 1}} \cdots h_{m}^{n_{m}} \right) = g
    \]
    is effectively $p$-automatic (recall that a union of effectively $p$-automatic sets is effectively $p$-automatic, see Lemma~\ref{lem:pauto}).
    But the above equation is equivalent to the Knapsack equation
    \begin{multline*}
        h_1^{n_1} \cdots h_{i_1 - 1}^{n_{i_1 - 1}} \left(g_1 h_{i_1 + 1} g_1^{-1}\right)^{n_{i_1 + 1}} \cdots \left( g_1 h_{i_2 - 1} g_1^{-1}\right)^{n_{i_2 - 1}} \left(g_1 g_2 h_{i_2 + 1} g_2^{-1} g_1^{-1}\right)^{n_{i_2 + 1}} \cdots \\
        \cdots \left(g_1 g_2 \dots g_s h_{m} g_s^{-1} \cdots g_2^{-1} g_1^{-1} \right)^{n_{m}} = g g_s^{-1} \cdots g_2^{-1} g_1^{-1}.
    \end{multline*}
    By Case 1, the solution set of the above equation is indeed effectively $p$-automatic.
\end{proof}

Combining the three steps in this section, we have proven Theorem~\ref{thm:main}:

\begin{proof}[Proof of~Theorem~\ref{thm:main}]
    By Proposition~\ref{prop:SMtoGP}, Proposition~\ref{prop:GPtoKP} and Theorem~\ref{thm:KP}, Submonoid Membership in $\mY \rtimes \Z^n$ reduces to checking whether the solution set of the Knapsack Equation~\eqref{eq:KP} is empty. Since the solution set is effectively $p$-automatic, its emptiness is decidable (Lemma~\ref{lem:pauto}).
\end{proof}

\section{S-unit equation over modules}\label{sec:Sunit}

The purpose of this section is to prove Theorem~\ref{thm:Sunit} on S-unit equations over modules.

\thmSunit*

As a comparison, S-unit equations over \emph{fields} are an intensively studied area in number theory. In the positive characteristic case, Adamczewski and Bell gave the following result. (Independently, Derksen and Masser~\cite{derksen2012linear} also proved a similar result.)
\begin{thrm}[{\cite[Theorem~3.1]{adamczewski2012vanishing}}]\label{thm:Sunitexample}
    Let $K$ be a field of characteristic $p > 0$, let $c_1, \ldots , c_m \in K^*$, and let $\ox = (x_1, \ldots, x_n) \in \left(K^*\right)^n$. Then the set of solutions $(\bz_1, \ldots, \bz_m) \in \left(\Z^n\right)^m$ for the equation
    \begin{equation*}
        \ox^{\bz_1} c_1 + \cdots + \ox^{\bz_m} c_m = 1
    \end{equation*}
    is an effectively $p$-automatic subset of $\left(\Z^n\right)^m \cong \Z^{nm}$.
\end{thrm}
Our objective is to generalize Theorem~\ref{thm:Sunitexample} from \emph{fields} to \emph{modules}. Notably, we need to replace the field $K$ by the polynomial ring $\F_p[\oX^{\pm}]$, and replace the elements $c_1, \ldots , c_m$ in $K^*$ by elements in the finitely presented $\F_p[\oX^{\pm}]$-module $M$. 
To give a general idea of our approach, consider the following example.

\begin{exmpl}[label=exa:Sunit]\label{expl:decomp}
    For simplicity, we present this example using regular polynomial rings instead of Laurent polynomial rings. Consider the equation
    \begin{equation}\label{eq:xayb}
        X^{a} + Y^{b} = 1
    \end{equation}
    in the $\F_2[X, Y]$-module $\F_2[X, Y]/Y^2(X+Y+1)$. Here, the variables are $a, b \in \N$.
    (Technically, in order to match the form of the equation in Theorem~\ref{thm:Sunit}, one needs to consider the equation $X^{a} Y^c + X^d Y^{b} = 1$ instead of~\eqref{eq:xayb}. However, to keep this example simple, we only consider the solutions with $c = d = 0$.)

    Equation~\eqref{eq:xayb} in $\F_2[X, Y]/Y^2(X+Y+1)$ is equivalent to $Y^2(X+Y+1) \mid X^{a} + Y^{b} + 1$ in $\F_2[X, Y]$, which in turn is equivalent to ``$Y^2 \mid X^{a} + Y^{b} + 1$ and $X+Y+1 \mid X^{a} + Y^{b} + 1$''.
    In other words, Equation~\eqref{eq:xayb} holds in $\F_2[X, Y]/Y^2(X+Y+1)$ if and only if it holds in both $\F_2[X, Y]/Y^2$ and in $\F_2[X, Y]/(X + Y + 1)$.

	The $\F_2[X, Y]$-module $\F_2[X, Y]/Y^2$ can be considered as a finitely generated $\F_2[X]$-module $\F_2[X, Y]/Y^2 = \F_2[X] + Y \cdot \F_2[X] \cong \F_2[X]^2$ by $1 \mapsto (1, 0), \; Y \mapsto (0, 1)$.
    Considered as an equation in $\F_2[X]^2$, the equation $X^{a} + Y^{b} = 1$ becomes 
\[
(X^a, 0) = (1, 0) \; \text{ if } \; b \geq 2, \quad
(X^a, 1) = (1, 0) \; \text{ if } \; b = 1, \quad
(X^a + 1, 0) = (1, 0) \; \text{ if } \; b = 0.
\]
	So the solution set of $X^{a} + Y^{b} = 1$ in $\F_2[X, Y]/Y^2$ is $\mZ_1 \coloneqq \{(a, b) \mid a = 0, \; b \geq 2\}$.
	
	The $\F_2[X, Y]$-module $\F_2[X, Y]/(X + Y + 1)$ can be considered as the $\F_2[X]$-module $\F_2[X]$ through the variable substitution $Y = X + 1$.
	Under this variable substitution, the equation $X^{a} + Y^{b} = 1$ becomes $X^{a} + (X+1)^b = 1$, so its solution set is $\mZ_2 \coloneqq \{(a, b) \mid a = b \in 2^{\N}\}$. This can be shown similar to Example~\ref{expl:KPtoSunit}, or by directly applying Theorem~\ref{thm:Sunitexample} on the quotient field $K = \F_2(X)$.
	
    Thus, the solution set of Equation~\eqref{eq:xayb} in $\F_2[X, Y]/Y^2(X+Y+1)$ is $\mZ_1 \cap \mZ_2 = \emptyset$.
    \hfill $\blacksquare$
\end{exmpl}

Example~\ref{expl:decomp} illustrates the idea that, to solve an equation in a module $M$, we first ``decompose'' it into several equations in \emph{coprimary modules} $M_1, \ldots, M_l$ (such as $\F_2[X, Y]/Y^2$ and $\F_2[X, Y]/(X + Y + 1)$). This will be formalized in Lemma~\ref{lem:inter}.
Then, for the equation in each coprimary module, we ``eliminate'' variables to simplify the quotients (for example $\F_2[X, Y]/Y^2 \cong \F_2[X]^2$ and $\F_2[X, Y]/(X + Y + 1) \cong \F_2[X]$ both eliminate the variable $Y$). In general, variable elimination might not be directly possible, and we need to perform a ``change of variables'' using \emph{Noether Normalization}. This will be formalized in Lemma~\ref{lem:Noether} and Lemma~\ref{lem:eff}.
Finally, we show that each such equation ``without quotient'' (more precisely, in \emph{torsion-free} modules) has a $p$-automatic solution set, by reducing to the case of fields (for example, considering equations over $\F_2[X]$ as equations over the quotient field $\F_2(X)$). This will be formalized in Lemma~\ref{lem:primary}.
Thus, the solution set of the original equation over $M$ is an intersection of effectively $p$-automatic sets, and is therefore also effectively $p$-automatic.
This approach is inspired by Derkson's proof of the Skolem-Mahler-Lech theorem in modules over positive characteristic~\cite[Chapter~9]{derksen2007skolem}.

The starting point of our proof is the following deep theorem by Adamczewski and Bell, generalizing Theorem~\ref{thm:Sunitexample}.

\begin{thrm}[{\cite[Theorem~4.1]{adamczewski2012vanishing}}]\label{thm:Sunitfield}
    Let $K$ be a field of characteristic $p > 0$ and let $D$ be a positive integer. Let $V$ be a Zariski closed subset of $\GL_D(K)$ and let $A_1, \ldots, A_N$ be commuting matrices in $\GL_D(K)$. Then the set $\{(z_1, \ldots, z_N) \in \Z^N \mid A_1^{z_1} A_2^{z_2} \cdots A_n^{z_N} \in V \}$ is effectively $p$-automatic.
\end{thrm}

For effectiveness reasons, we need to suppose $K$ to be an effectively computable field (field where all arithmetic operations are effective): this will be the case in all our applications.
We will not give the formal definition of a \emph{Zariski closed subset} because in our application of Theorem~\ref{thm:Sunitfield}, the set $V$ will be of the form
\begin{equation}\label{eq:linearV}
V = \{A \in \GL_D(K) \mid BAv = 0^d\}
\end{equation}
for some $d \in \N, B \in K^{d \times D}, v \in K^{D \times 1}$. This is a linear subset of $\GL_D(K)$ and hence Zariski closed. 
For a tuple of commutative matrices $\oA = (A_1, \ldots, A_n) \in \left(\GL_d(K)\right)^n$ and a vector $\bz = (z_1, \ldots, z_n) \in \Z^n$, denote
\[
\oA^{\bz} \coloneqq A_1^{z_1} A_2^{z_2} \cdots A_n^{z_n},
\]
similar to our notation for monomials.
In the case where $V$ has the form~\eqref{eq:linearV}, Theorem~\ref{thm:Sunitfield} gives the following corollary, which can be considered as a version of Theorem~\ref{thm:Sunit} where the module $M$ is replaced by a vector space over a field $K$.

\begin{cor}\label{cor:MordellLang}
    Let $K$ be a field of characteristic $p > 0$ and let $\oA = (A_1, \ldots, A_n) \in \left(\GL_d(K)\right)^n$ be a tuple of commuting matrices.
    Let $c_0, c_1, \ldots, c_m$ be elements in the $d$-dimensional vector space $K^d$.
    Then, the set of solutions $(\bz_1, \ldots, \bz_m) \in \left(\Z^{n}\right)^{m}$ for the equation
    \begin{equation}\label{eq:Sunitfield}
        \oA^{\bz_1} c_1 + \cdots + \oA^{\bz_m} c_m = c_0
    \end{equation}
    is an effectively $p$-automatic subset of $\left(\Z^{n}\right)^{m} \cong \Z^{nm}$.
\end{cor}
\begin{proof}
    For any $B_1, \ldots, B_{m+1} \in \GL_d(K)$, we denote by $\diag(B_1, \ldots, B_{m+1})$ the block diagonal matrix in $\GL_{(m+1)d}(K)$ whose block diagonal entries are $B_1, \ldots, B_{m+1}$.
    Let $I$ denote the identity matrix of dimension $d$ and denote $D = (m+1)d$.
    Consider the following elements in $\GL_{D}(K)$.
    For $j = 1, \ldots, n$, let $A_{1j} \coloneqq \diag(A_j, I, \ldots, I), A_{2j} = \diag(I, A_j, \ldots, I), \ldots, A_{mj} = \diag(I, \ldots, A_j, I)$.  
    
    Furthermore, let $v \coloneqq (c_1^{\top}, c_2^{\top}, \ldots, c_m^{\top}, c_0^{\top})^{\top} \in K^{D \times 1}$, and $B \coloneqq (I, I, \ldots, I, -I) \in K^{d \times D}$.
    Then, Equation~\eqref{eq:Sunitfield} is equivalent to
    \begin{equation}\label{eq:prodab}
    B\left( \prod_{i = 1}^m \prod_{j = 1}^n A_{ij}^{z_{ij}} \right)v = 0^d.
    \end{equation}
    We then apply Theorem~\ref{thm:Sunitfield} to the set
    \[
    V = \{A \in \GL_D(K) \mid BAv = 0^d\}
    \]
    and the commuting matrices $A_{ij}, i = 1, \ldots, m; j = 1, \ldots, n$.
    This shows that the set of solutions $(z_{ij})_{1 \leq i \leq m, 1 \leq j \leq n}$ for Equation~\eqref{eq:prodab} is an effectively $p$-automatic subset of $\Z^{nm}$.
\end{proof}

%
To pass from vector spaces to modules, we now introduce the \emph{primary decomposition} of a module.
First, we recall some standard definitions and tools from commutative algebra~\cite{eisenbud2013commutative}.

\begin{defn}\label{def:commalg}
    Let $R$ be a commutative Noetherian ring (for example, $R = \K[\oX^{\pm}]$ for some field $\K$).
    \begin{enumerate}[(1)]
        \item An \emph{ideal} of $R$ is an $R$-submodule of $R$. An ideal $I \subseteq R$ is called \emph{prime} if $I \neq R$, and for every $a, b \in R$, $ab \in I$ implies $a \in I$ or $b \in I$. Prime ideals are usually denoted by the Gothic letter $\frp$.
        \item Let $M$ be a finitely generated $R$-module. The \emph{annihilator} of an element $m \in M$, denoted by $\Ann_R(m)$, is the set $\{r \in R \mid r\cdot m = 0\}$.
        \item An $R$-module $M$ is called \emph{torsion-free} if for every $r \in R, m \in M$, $r \cdot m = 0$ implies $r = 0$ or $m = 0$. That is, $M$ is torsion-free if and only if $\Ann_R(m) = \{0\}$ for all $m \in M \setminus \{0\}$.
        \item A prime ideal $\frp \subset R$ is called \emph{associated} to $M$ if there exists a non-zero $m \in M$ such that $\frp = \Ann_R(m)$. 
        \item Let $N$ be a finitely generated $R$-module. A submodule $N'$ of $N$ is called \emph{primary} if $N/N'$ has only one associated prime ideal. If we denote this prime ideal by $\frp$, then $N' \subseteq N$ is called \emph{$\frp$-primary}.
        \item Let $N'$ be a submodule of a finitely generated $R$-module $N$. The \emph{primary decomposition} of $N'$ is the writing of $N'$ as a finite intersection $\bigcap_{i = 1}^l N_i$, where $N_i$ is a $\frp_i$-primary submodule of $N$ for some prime ideal $\frp_i \subset R$.
        A primary decomposition always exists~\cite[Theorem~3.10]{eisenbud2013commutative}.
        If $R = \F_p[\oX^{\pm}]$ and $N, N'$ are finitely generated submodules of $R^d$ for some $d \in \N$, then a primary decomposition of $N' \subseteq N$ can be effectively computed~\cite{rutman1992grobner}.       
        \item A finitely generated $R$-module $M$ is called \emph{coprimary} if the submodule $\{0\}$ is primary, that is, if $M$ has only one associated prime ideal. If we denote this prime ideal by $\frp$, then $M$ is called \emph{$\frp$-coprimary}.
        If $M$ is $\frp$-coprimary, and $m$ is a non-zero element in $M$, then $\Ann_R(m) \subseteq \frp$.
    \end{enumerate}
\end{defn}

Let $M = \F_p[\oX^{\pm}]^d/N, \; N \subseteq \F_p[\oX^{\pm}]^d$ be the finite presentation of $M$ in Theorem~\ref{thm:Sunit}.
Let $N = \bigcap_{i = 1}^l N_i$ be the primary decomposition of $N$ as a submodule of $\F_p[\oX^{\pm}]^d$, where $N_i$ is $\frp_i$-primary for a prime ideal $\frp_i \subset \F_p[\oX^{\pm}]$, $i = 1, \ldots, l$.
Then $M_i \coloneqq \F_p[\oX^{\pm}]^d/N_i$ is $\frp_i$-coprimary.
The finite presentation of $N_i$, and hence $M_i$, can be effectively computed (Definition~\ref{def:commalg}(6)).
Since $N \subseteq N_i$, there is a canonical map $\rho_i \colon M = \F_p[\oX^{\pm}]^d/N \rightarrow M_i = \F_p[\oX^{\pm}]^d/N_i$.
Since $N = \bigcap_{i = 1}^l N_i$, the intersection of kernels $\bigcap_{i=1}^l \ker(\rho_i)$ is $\{0\}$.


\begin{lem}\label{lem:inter}
    For $i = 1, \ldots, l$, let $\mZ_i$ denote the set of solutions $(\bz_1, \ldots, \bz_m) \in \left(\Z^n\right)^m$ of the following equation in $M_i$:
    \begin{equation}\label{eq:Suniti}
        \oX^{\bz_1} \cdot \rho_i(c_1) + \cdots + \oX^{\bz_m} \cdot \rho_i(c_m) = \rho_i(c_0).
    \end{equation}
    Then the solution set of Equation~\eqref{eq:Sunitori} is exactly the intersection $\bigcap_{i = 1}^l \mZ_i$.
\end{lem}
\begin{proof}
    Since $\bigcap_{i=1}^l \ker(\rho_i) = \{0\}$, an element $y \in M$ is zero if and only if $\rho_i(y) = 0$ for all $i = 1, \ldots, l$.
    Equation~\eqref{eq:Sunitori} is equivalent to 
    $
    \oX^{\bz_1} \cdot c_1 + \cdots + \oX^{\bz_m} \cdot c_m - c_0 = 0
    $.
    This holds if and only if 
    $
    \rho_i\left(\oX^{\bz_1} \cdot c_1 + \cdots + \oX^{\bz_m} \cdot c_m - c_0\right) = 0
    $
    for all $i = 1, \ldots, l$.
    Therefore, Equation~\eqref{eq:Sunitori} holds if and only if Equation~\eqref{eq:Suniti} holds for all $i$.
    Hence, the solution set of Equation~\eqref{eq:Sunitori} is exactly $\bigcap_{i = 1}^l \mZ_i$.
\end{proof}

\renewcommand\thmcontinues[1]{continued}
\begin{exmpl}[continues=exa:Sunit]
In Example~\ref{expl:decomp}, we took $R$ to be the ring $\F_2[X, Y]$, and $M$ to be the finitely presented $R$-module $\F_2[X, Y]/Y^2(X+Y+1)$. In this case, we have $M = \F_2[X, Y]/N$ with $N = \F_2[X, Y] \cdot Y^2(X+Y+1)$, whose primary decomposition is $N = N_1 \cap N_2, \; N_1 = \F_2[X, Y] \cdot Y^2, \, N_2 = \F_2[X, Y] \cdot (X+Y+1)$. The modules $M_1 = \F_2[X, Y]/N_1 = \F_2[X, Y]/Y^2$ and $M_2 = \F_2[X, Y]/N_2 = \F_2[X, Y]/(X+Y+1)$ are respectively $\frp_1$ and $\frp_2$-coprimary, where $\frp_1$ is the ideal of $R$ generated by $Y$, and $\frp_2$ is the ideal of $R$ generated by $X+Y+1$.
The sets $\mZ_1, \mZ_2$ are respectively solutions of the equation $X^a + Y^b = 1$ in $M_1$ and $M_2$.
    \hfill $\blacksquare$
\end{exmpl}

Our next step is to show that each $\mZ_i$ is effectively $p$-automatic.
If so, then their intersection $\bigcap_{i = 1}^l \mZ_i$ will also be $p$-automatic (see Lemma~\ref{lem:pauto}).
Fix $i \in \{1, \ldots, l\}$. Note that Equation~\eqref{eq:Suniti} is now an equation over the $\frp_i$-coprimary $\F_p[\oX^{\pm}]$-module $M_i$.
This motivates us to consider the quotient ring $\F_p[\oX^{\pm}]/\frp_i$.
We start with the following staple result in commutative algebra, which can be intuitively understood as performing a ``change of variable'' to simplify $\F_p[\oX^{\pm}]/\frp_i$.

\begin{lem}[{Noether Normalization Lemma~\cite[p.2]{mumford2004red}}]\label{lem:Noether}
    Let $\K$ be an infinite field and $A$ be a finitely generated $\K$-algebra. Then there exist algebraically independent elements $\widetilde{X}_1, \ldots, \widetilde{X}_s \in A$ such that $A$ is a finitely generated $\K[\widetilde{X}_1, \ldots, \widetilde{X}_s]$-module.
\end{lem}

Ideally we would want to apply the Noether Normalization Lemma to $A = \F_p[\oX^{\pm}]/\frp_i$. Even though $\F_p$ is not an infinite field, we can without loss of generality replace it with its algebraic closure $\K$ (which is infinite).
Formally, we replace $\frp_i$ with $\frp_i \otimes_{\F_p} \K$ and $M_i$ with $M_i \otimes_{\F_p} \K$.
Now $\frp_i$ and $M_i$ are respectively an ideal and a module over the ring $\K[\oX^{\pm}] = \F_p[\oX^{\pm}] \otimes_{\F_p} \K$. Note that this does not change the solution set $\mZ_i$.

The classical proof of Noether Normalization Lemma is constructive~\cite[p.2]{mumford2004red}.
In our context, this means that given $A = \K[\oX^{\pm}]/\frp_i$, the lemma explicitly gives the expressions for $\widetilde{X}_1, \ldots, \widetilde{X}_s \in A$ as elements in $\K[\oX^{\pm}]/\frp_i$.

Choose $Y_1, \ldots, Y_s \in \K[\oX^{\pm}]$ such that $Y_i + \frp_i = \widetilde{X}_i$ for all $i$.
Since $\widetilde{X}_1, \ldots, \widetilde{X}_s$ are algebraically independent in $A = \K[\oX^{\pm}]/\frp_i$, we have $\frp_i \cap \K[Y_1, \ldots, Y_s] = \{0\}$.
Any finitely generated $\K[\oX^{\pm}]$-module $M$ is also a $\K[Y_1, \ldots, Y_s]$-module. However, $M$ might not be finitely generated as a $\K[Y_1, \ldots, Y_s]$-module.
However, in case $M$ is also finitely generated as a $\K[Y_1, \ldots, Y_s]$-module, then a finite presentation can be effectively computed.
This is summarized in the following rather standard result in effective commutative algebra.


\begin{lem}[{\cite[Section~2]{baumslag1981computable}}]\label{lem:eff}
    Let $Y_1, \ldots, Y_s \in \K[\oX^{\pm}]$.
    Suppose we are given the finite presentation of an $\K[\oX^{\pm}]$-module $M$, as well as a finite number of generators $m_1, \ldots, m_t$ of $M$ as a $\K[Y_1, \ldots, Y_s]$-module.
    Then one can effectively compute the finite presentation of $M$ as a $\K[Y_1, \ldots, Y_s]$-module.
\end{lem}
\begin{proof}
    Let $M = \K[\oX^{\pm}]^d/N$ be its finite presentation as a $\K[\oX^{\pm}]$-module.
    Let $f_1, \ldots f_s$ be polynomials in $\K[\oX^{\pm}]$ that correspond to the elements $Y_1, \ldots Y_s$.
    Then $M$ is also an effectively finitely presented module over the polynomial ring $\K[\oX^{\pm}, Y_1, \ldots, Y_s]$ by 
    $M \cong \K[\oX^{\pm}, Y_1, \ldots, Y_s]^d/\big(\K[\oX^{\pm}, Y_1, \ldots, Y_s] \cdot N + \sum_{i = 1}^s (Y_i - f_i) \cdot \K[\oX^{\pm}, Y_1, \ldots, Y_s]^d \big)$.
    Then, we can compute the finite presentation of the $\K[Y_1, \ldots, Y_s]$-submodule of $M$ generated by $m_1, \ldots, m_t$ (see~\cite[Corollary~2.5]{baumslag1981computable}).
    This submodule is exactly $M$.
\end{proof}

\renewcommand\thmcontinues[1]{continued}
\begin{exmpl}[continues=exa:Sunit]
Consider the ideal $\frp_1$ of $\K[X, Y]$ generated by $Y$.
Let $A$ be the ring $\K[X, Y]/\frp_1 = \K[X, Y]/Y$.
Then in the Noether Normalization Lemma we can let $s = 1$, and $\widetilde{X}_1 = X + \frp_1$. So $\K[X, Y]/Y \cong \K[\widetilde{X}_1]$ is a $\K[\widetilde{X}_1]$-module.
Let $Y_1 = X_1$, then the $\K[X, Y]$-module $M_1 = \K[X, Y]/Y^2$ is a finitely presented $\K[Y_1]$-module $\K[Y_1]^2$.
%
    \hfill $\blacksquare$
\end{exmpl}

\begin{lem}\label{lem:primary}
    The set $\mZ_i$ is effectively $p$-automatic.
\end{lem}
\begin{proof}
	The main line of our proof follows~\cite[Lemma~9.7-9.8]{derksen2007skolem}, with additional remarks given on effectiveness. 
	As above, we can replace the base field $\F_p$ by its algebraic closure $\K$. 

    Since $M_i$ is $\frp_i$-coprimary, we can compute $t \in \N$ such that $\frp_i^t M_i = 0$ (such $t$ exists by the ascending chain condition~\cite[Chapter~1.4]{eisenbud2013commutative}).
    We have
    $
    M_i \supseteq \frp_i M_i \supseteq \frp_i^2 M_i \supseteq \cdots \supseteq \frp_i^t M_i = 0
    $.
    Because $\frp_i \cdot (\frp_i^{j} M_i/\frp_i^{j+1} M_i) = 0$, each quotient $\frp_i^{j} M_i/\frp_i^{j+1} M_i$ is a finitely generated $\K[\oX^{\pm}]/\frp_i$-module with effectively computable generators, and therefore a finitely generated $\K[Y_1, \ldots, Y_s]$-module with effectively computable generators.
    It follows that the $\K[\oX^{\pm}]$-module $M_i$ is also finitely generated as a $\K[Y_1, \ldots, Y_s]$-module, with effectively computable generators.
    Since the finite presentation of $M_i$ as a $\K[\oX^{\pm}]$-module is given, we can compute the finite presentation of $M_i$ as a $\K[Y_1, \ldots, Y_s]$-module by Lemma~\ref{lem:eff}.

    Since $M_i$ is $\frp_i$-coprimary, the annihilator of any non-zero element in $M_i$ as a $\K[Y_1, \ldots, Y_s]$-module is contained in $\frp_i \cap \K[Y_1, \ldots, Y_s] = \{0\}$ (see Definition~\ref{def:commalg}(7)). Therefore, $M_i$ is a finitely generated torsion-free $\K[Y_1, \ldots, Y_s]$-module (see Definition~\ref{def:commalg}(3)). 
    
    Consider the quotient field $K \coloneqq \K(Y_1, \ldots, Y_s)$ of $\K[Y_1, \ldots, Y_s]$. The localization $\widetilde{M_i} \coloneqq M_i \otimes_{\K[Y_1, \ldots, Y_s]} K$ is a finite dimensional $K$-linear space, so $\widetilde{M_i} \cong K^d$ for some integer $d$.
    Because $M_i$ is torsion-free, the canonical map $\varphi \colon M_i \rightarrow M_i \otimes_{\K[Y_1, \ldots, Y_s]} K$ is injective.
    The linear transformations $A_j \colon \widetilde{M_i} \rightarrow \widetilde{M_i}, y \mapsto X_j \cdot y,$ commute for $j = 1, \ldots, n$, and they are respectively inverses of the transformations $A_j^{-1} \colon \widetilde{M_i} \rightarrow \widetilde{M_i}, y \mapsto X_j^{-1} \cdot y$. Hence, $A_j$ is in $\GL_d(K)$ for $j = 1, \ldots, n$, for a fixed basis of $\widetilde{M_i} \cong K^d$, and these matrices commute.

    Since $\varphi$ is injective, Equation~\eqref{eq:Suniti} holds if and only if
    \begin{equation}\label{eq:Sunitlocal}
    \oA^{\bz_1} \cdot \varphi(\rho_i(c_1)) + \cdots + \oA^{\bz_m} \cdot \varphi(\rho_i(c_m)) = \varphi(\rho_i(c_0)),
    \end{equation}
    where $\oA = (A_1, \ldots, A_n)$.
    So $\mZ_i$ is the solution set for Equation~\eqref{eq:Sunitlocal}.
    Note that $K = \K(Y_1, \ldots, Y_s)$ is an effective field. We now apply Corollary~\ref{cor:MordellLang}, and conclude that the solution set for Equation~\eqref{eq:Sunitlocal} is an effectively $p$-automatic subset of $\left(\Z^{n}\right)^{m} \cong \Z^{nm}$. 
\end{proof}

\renewcommand\thmcontinues[1]{continued}
\begin{exmpl}[continues=exa:Sunit]
As above, the $\K[X, Y]$-module $M_1 = \K[X, Y]/Y^2$ is isomorphic to the $\K[X]$-module $\K[X]^2$ by $\K[X, Y]/Y^2 \xrightarrow[]{\sim} \K[X]^2, \, f_0 + Y f_1 + Y^2 f_2 + \cdots \mapsto (f_0, f_1)$, for $f_0, f_1, f_2, \ldots \in \K[X]$.

We would like to solve the equation $X^a + Y^b = 1$ in $M_1$.
Let $\K(X)$ denote the quotient field of $\K[X]$: this is the set of univariate rational functions over the field $\K$.
Note that $\varphi: M_1 \cong \K[X]^2 \hookrightarrow \K(X)^2$ is injective.
Consider $\K(X)^2$ as a 2-dimensional vector space over $\K(X)$, then the map $\varphi(m) \mapsto \varphi(X \cdot m), m \in M_1,$ extends to a linear transformation associated to the matrix 
$
A_X =
\begin{pmatrix}
    X & 0 \\
    0 & X \\
\end{pmatrix}
$,
and the map $\varphi(m) \mapsto \varphi(Y \cdot m), m \in M_1$ extends to a linear transformation associated to the matrix 
$
A_Y =
\begin{pmatrix}
    0 & 1 \\
    0 & 0 \\
\end{pmatrix}
$.
(Here, $A_Y$ is not invertible because the variable $Y$ is not invertible in $\K[X, Y]$. However in the proof of Lemma~\ref{lem:primary}, we are considering Laurent polynomial rings, so all the associated matrices would be invertible.)

As a result, the equation $X^a + Y^b = 1$ in $M_1$ is equivalent to the equation
$
A_X^a \cdot (1, 0)^{\top} + A_Y^b (1, 0)^{\top} = (1, 0)^{\top}
$
in the 2-dimensional vector space $\K(X)^2$.
Subject to the invertibility of $A_Y$, we can then use Corollary~\ref{cor:MordellLang} to conclude that the equation admits an effectively $2$-automatic solution set.
\hfill $\blacksquare$
\end{exmpl}

Combining Lemma~\ref{lem:inter} and Lemma~\ref{lem:primary}, we obtain a proof of Theorem~\ref{thm:Sunit}:
\begin{proof}[Proof of Theorem~\ref{thm:Sunit}]
    By Lemma~\ref{lem:inter}, the solution set of Equation~\eqref{eq:Sunitori} is the intersection $\bigcap_{i = 1}^l \mZ_i$.
    Lemma~\ref{lem:primary} shows that each $\mZ_i$ is effectively $p$-automatic, so their intersection is also effectively $p$-automatic (Lemma~\ref{lem:pauto}).
\end{proof}

\bibliography{lamplighter}

\end{document}